\documentclass[10pt,a4paper,onecolumn,oneside]{article}
\usepackage[utf8]{inputenc}
\usepackage[english]{babel}

\usepackage[top=2cm,bottom=2cm,left=3cm,right=3cm]{geometry} 
\usepackage{setspace}
\usepackage{enumitem}

\usepackage{amsmath,amsfonts,amssymb,amsthm,mathrsfs}
\allowdisplaybreaks[1]

\usepackage{hyperref}
\hypersetup{colorlinks=true,linkcolor=red!70!green!70!black,citecolor=red!70!green!70!black,urlcolor=magenta}
\usepackage{graphicx,color,xcolor,array}

\newtheorem{theorem}{Theorem}[section]
\newtheorem{lemma}[theorem]{Lemma}
\newtheorem{proposition}[theorem]{Proposition}
\newtheorem{corollary}[theorem]{Corollary}
\numberwithin{equation}{section}

\newcommand{\el}[1]{\rule[#1]{0pt}{0pt}} 

\def\CC{\mathscr{C}}
\def\esup{\text{ess\,sup}}
\def\id{\text{id}}
\def\NN{\mathbb{N}}
\def\RR{\mathbb{R}}

\newcommand{\LL}[1]{L^{#1}(\Omega)}
\newcommand{\LLL}[1]{L^{#1}(\Omega,drdz)}
\newcommand{\woo}[1]{\|#1\omega_0\|_{L^1(\Omega)}}
\newcommand{\woorz}[1]{\|#1\omega_0(r,z)\|_{L^1(\Omega)}}
\newcommand{\woop}[2]{\|#1\omega_0\|_{L^{#2}(\Omega)}}
\newcommand{\wws}[1]{\|#1\omega_\theta(s)\|_{L^1(\Omega)}}
\newcommand{\wwsp}[2]{\|#1\omega_\theta(s)\|_{L^{#2}(\Omega)}}
\newcommand{\wwt}[1]{\|#1\omega_\theta(t)\|_{L^1(\Omega)}}
\newcommand{\wwtrz}[1]{\|#1\omega_\theta(t,r,z)\|_{L^1(\Omega)}}
\newcommand{\wwtp}[2]{\|#1\omega_\theta(t)\|_{L^{#2}(\Omega)}}
\newcommand{\wwtt}[2]{\|#1\omega_\theta(#2)\|_{L^1(\Omega)}}
\newcommand{\wwttp}[3]{\|#1\omega_\theta(#2)\|_{L^{#3}(\Omega)}}

\newcommand{\di}{\textup{div}}
\newcommand{\dii}{\textup{div}_*}
\newcommand{\uus}{\|u(s)\|_{L^\infty(\Omega)}}
\newcommand{\uut}{\|u(t)\|_{L^\infty(\Omega)}}
\newcommand{\rot}{\textup{rot}}
\newcommand{\w}{\omega_\theta}

\newcommand{\wo}{\omega_0}

\begin{document}

\author{Quentin \textsc{Vila}\footnote{
Institut Fourier (CNRS), UMR 5582; Univesité Grenoble Alpes, CS40700, 38058 Grenoble cedex 09, France} 
}
\title{ \LARGE \bf
Time-Asymptotic Study of a Viscous Axisymmetric Fluid without Swirl
}
\date{{\normalsize the 20$^{th}$ May 2022}}


\maketitle

\begin{center}
{\bf Abstract.}
\end{center}

We study the long-time behaviour of axisymmetric solutions without swirl for the three-dimensional Navier-Stokes equations in the whole space. Assuming that the initial vorticity is sufficiently localised, we compute explicitly the leading terms in the asymptotic expansion of the solution, both for the vorticity and the velocity field. In particular, we identify optimal temporal decay rates depending on the spatial localisation of the initial data. Our approach relies on accurate $L^p$-$L^q$ estimates for the linearised evolution equation and its Taylor expansion in self-similar variables. 



\section{Introduction}

In absence of nonconservative forces, the Navier-Stokes equations for homogeneous and incompressible fluids can be expressed in dimensionless variables as the system
\begin{equation}\label{NSdimensionless}
\left\{\begin{array}{l}
\partial_tu + (u.\nabla)u = -\nabla P + \Delta u\\
\rule[1.3em]{0pt}{0pt}
\di\,u = 0
\end{array}\right.
\end{equation}
whose unknown is the pair $(u,P)$, where $u$ is the (dimensionless) velocity field and $P$ the (dimensionless) pressure field of the fluid. The present paper supposes the studied fluid to be \emph{axisymmetric}, meaning that the fluid is evolving in the three-dimensional real space in such a way that its motion is invariant by rotation around the vertical axis. Considering the cylindrical coordinates system $(r,\theta,z)$ and its associated basis
\[e_r = \begin{pmatrix}\cos\theta\\\sin\theta\\0\end{pmatrix}, \qquad e_\theta = \begin{pmatrix}-\sin\theta\\\cos\theta\\0\end{pmatrix}, \qquad e_z = \begin{pmatrix}0\\0\\1\end{pmatrix},\]
it means that the cylindrical components of the velocity $u_r = u\cdot e_r$, $u_\theta = u\cdot e_\theta$ and $u_z = u\cdot e_z$ are independant of $\theta$: $u_r(t,x) = u_r(t,r,\theta,z) = u_r(t,r,z)$ for all $x = (r\cos\theta,r\sin\theta,z)\in\RR^3$ and likewise for $u_\theta$ and $u_z$. The tangential component of the velocity field $u_\theta$ is named the \emph{swirl}. In this paper we also make the hypothesis that the swirl is identically equal to zero, so
\[
u(t,r,\theta,z) = u_r(t,r,z)e_r + u_z(t,r,z)e_z.
\]
To a certain extent, the axisymmetric fluid without swirl is comparable to a two-dimensional fluid admitting an axial symmetry. 

The first mathematical results in axisymmetric fluid mechanics date back to 1968, when Ladyzhenskaya \cite{LadyzhenUni} and Ukhovskii and Yudovich \cite{UkhovYudo} tackle the uniqueness problem of the three-dimensional Navier-Stokes equations by adding this hypothesis. Comes later the study of Leonardi, M\'alek, Ne\v cas and Pokorn\'y \cite{LMNPaxialflow}, who show the global existence of a zero-swirl axisymmetric solution to \eqref{NSdimensionless} from any such initial data in $H^2(\RR^3)$, then the one of Abidi \cite{Abidi2008} who improved the result for an initial data in $H^{\frac{1}{2}}(\RR^3)$. Follows finally the recent paper of Gallay and \v Sver\'ak \cite{Gallay1}, suggesting to work in some two-dimensional-like $L^p$ spaces equipped with scale-invariant norms. 

The present article examines in detail what can be said about the long-time behaviour of the axisymmetric fluid when the initial vorticity is well localised. Using tools specific to axisymmetry without swirl, we determine the asymptotic expansions of the fluid's vorticity and velocity under essentially optimal hypotheses. The interest for the long-time behaviour of the Navier-Stokes flow in the litterature is substantial, and intensified in the end of the last century. Starting in the 1980's with the papers of Kato \cite{Kato84}, Schonbek \cite{Schonbek85,Schonbek86} and Wiegner \cite{Wiegner87}, a series of important works have addressed the time-decay problem for fluids in $\RR^3$, or $\RR^n$, among which can be cited \cite{Carpio94, Carpio96,MiyakawaSchonbek2001,FujigakiMiyakawa2001,GalWayR3,
Brandolese2004,Brandolese2003,OkaTsu2016,BrandOka2021}. For an extensive overview on the subject, we invite the reader to look at the introduction of the most recent ones \cite{OkaTsu2016,BrandOka2021} and the references therein, or to refer to the chapter by Brandolese and Schonbeck in \cite{GigaNovo2018}.

The hypothesis of axisymmetry without swirl encourages the consideration of the vorticity $\omega$, which is the curl of the velocity $\omega = \rot(u)$, as the core of the problem. Indeed, in this context, the vorticity boils down to a real-valued function $\w$ by writing that $\omega = \w e_\theta$ where $\w = \partial_zu_r - \partial_ru_z$. Given the Navier-Stokes equations \eqref{NSdimensionless} and the zero-swirl axisymmetry hypothesis, $\w$ can be found as the solution of the equation
\begin{equation}\label{voraxi}
\left\{\begin{array}{ll} \displaystyle
\partial_t\w + u.\nabla\w - \frac{u_r}{r}\w = \Delta\w - \frac{\w}{r^2}
& \text{on }]0,+\infty[\times\Omega
\\ 
\w(t,0,z) = 0 & \forall(t,0,z)\in[0,+\infty[\times\partial\Omega
\\ \rule[1.3em]{0pt}{0pt}
\w(0,r,z) = \wo(r,z) & \forall(r,z)\in\Omega
\end{array}\right.
\end{equation}
where $\Omega$ is the open half plane $\Omega =\ ]0,+\infty[\times\RR$ and where $\w=0$ at $r=0$ due to the regularity of $u$. Here, in cylindrical coordinates, $u\cdot\nabla = u_r\partial_r + u_z\partial_z$ and $\Delta = \partial_r^2 + \frac{1}{r}\partial_r + \partial_z^2$. One notable advantage of working with the vorticity is that the pressure $P$ is not present any more in the equation, yet no generality is lost thanks to the Biot-Savart law. Deeper discussions about the relevance to work in vorticity can also be found in \cite{GalWayR2,GalWayR3}. 

Following \cite{Gallay1}, our study will be carried out using the scale-invariant norms
\[
\|w\|_{\LL{p}} = \left(\int_\Omega |w(r,z)|^p\,drdz\right)^{\frac{1}{p}} \quad \text{ if }1\leqslant p<\infty \quad ; \quad
\|w\|_{\LL{\infty}} = \underset{\Omega}{\esup}\,|w|,
\]
according to the two-dimensional Lebesgue measure $drdz$ on the half plane $\Omega$ instead of the three-dimensional Lebesgue measure on $\RR^3$. Up to a multiplicative factor, $\|w\|_{\LL{p}}$ is the norm of $(r,z)\mapsto r^{-\frac{1}{p}}w(r,z)$ in the usual $L^p(\RR^3,rdrd\theta dz)$ space. Within this framework, it is known thanks to \cite{Gallay1} that the zero-swirl axisymmetric problem \eqref{voraxi} is well-posed for any initial data $\wo$ in $\LL{1} = \LLL{1}$ and that moreover, unless $\wo$ is trivial, the $L^1$-norm of the solution $\wwt{} = \int_\Omega|\w(t,r,z)|drdz$ is a decreasing function of time and tends to zero as $t$ tends to infinity. We here go even further, naming
\begin{equation}\label{defIJ}
\begin{array}{c}\displaystyle
I(t) = \int_\Omega r^2\,\w(t,r,z)\,drdz \quad\text{and}\quad J(t) = \int_\Omega r^2z\,\w(t,r,z)\,drdz
\end{array}
\end{equation}
then stating the following theorem.

\begin{theorem}\label{theoremintro}
Consider $\w$ the solution of \eqref{voraxi} with initial data $\wo\in\LLL{1}$.\\
If $\woo{r^2}<\infty$ then $I(t)$ is constantly equal to $I_0 = I(0)$ and the relation
\begin{equation}\label{theoremintror2}
\w(t,r,z) = \frac{rI_0}{16\sqrt{\pi}}e^{\mbox{$-\frac{r^2+z^2}{4t}$}}\ t^{-5/2} + \underset{t\to+\infty}{o}\left(t^{\frac{1}{p}-2}\right)
\end{equation}
holds in $\LLL{p}$ for every $p\in[1,+\infty]$.\\
If moreover $\woo{(r^3+r^2|z|+r|z|)}<\infty$, then $J(t)$ has a finite limit $J_\infty$ as $t\to+\infty$ and
\begin{equation}\label{theoremintror3}
\w(t,r,z) = \frac{rI_0}{16\sqrt{\pi}}e^{\mbox{$-\frac{r^2+z^2}{4t}$}}\ t^{-5/2} + \frac{rzJ_\infty}{32\sqrt{\pi}}e^{\mbox{$-\frac{r^2+z^2}{4t}$}}\ t^{-7/2} + \underset{t\to+\infty}{o}\left(t^{\frac{1}{p}-\frac{5}{2}}\right)
\end{equation}
holds in $\LLL{p}$ for every $p\in[1,+\infty]$.
\end{theorem}

\paragraph{Remarks.}
\begin{itemize}[leftmargin=0pt,itemindent=!]
\item[1.] Theorem \ref{theoremintro} is the concatenation of the two central results of this paper, propositions \ref{propwr2} and \ref{propwr3}, and of remark \ref{rkJinfty}.
\item[2.] Expansion \eqref{theoremintror2} already appears in \cite{Gallay1}, but with the additional hypothesis that $\wo$ is nonnegative. Let us note especially that \eqref{theoremintror2} gives for $p=1$ the general estimate $\wwt{} \leqslant C_{\wo}t^{-1}$, where $C_{\wo}$ is a positive constant that depends only on the initial data $\wo$.
\item[3.] Expansion \eqref{theoremintror3} is a new result for zero-swirl axisymmetry. Its terms could have been obtained from the general three-dimensional study developed in \cite{GalWayR3}, but \eqref{theoremintror3} is here deduced from much weaker assumptions thanks to the specific context of zero-swirl axisymmetry. We discuss this point more precisely in appendix \ref{annexecomparaisonThGWay3D}.
\item[4.] 
The technical lemmas that lead to theorem \ref{theoremintro} give more information than just showing the terms of the asymptotic expansion: they provide estimates on the $L^p$-norms of the vorticity for all the intermediate hypotheses on its initial moments. Indeed, 
\begin{itemize}[leftmargin=\parindent,itemindent=!,topsep=0pt] 
\item proposition \ref{propmajorationsralpha} treats the cases where $\wo\in\LL{1}$ is such that $\woo{r^\alpha}<\infty$ with $0\leqslant\alpha\leqslant2$, 
\item proposition \ref{propwr2+alpha} the cases where moreover $\woo{r^\alpha(r^2+r|z|+|z|)}<\infty$, $0<\alpha\leqslant1$,
\item proposition \ref{propwr3+alpha} the cases where $\woo{r^\alpha(r^3+rz^2+z^2)}<\infty$, $0<\alpha<1$,
\item and proposition \ref{propwr4} treats the case where $\woo{(r^4+r^2z^2+rz^2)}<\infty$.
\end{itemize}
These increasingly precise hypotheses on the finite moments of $\wo$ are the right amount of localisation to ask in order to refine our observation of the vorticity's behaviour. This echoes the well-known fact that there exists a link between the spatial and the temporal decays at infinity, for both the velocity (\textit{cf} \cite{Schonbek86,Carpio96,
MiyakawaSchonbek2001,FujigakiMiyakawa2001,Brandolese2004, Brandolese2003,OkaTsu2016}) and the vorticity (\textit{cf} \cite{GalWayR2,GalWayR3}).
\end{itemize}

Theorem \ref{theoremintro} has a translation in terms of velocity. When $\w$ is solution of \eqref{voraxi}, $u$ is the solution of the system
\begin{equation}\label{NSaxy}
\left\{\begin{array}{ll}
\partial_tu + (u\cdot\nabla)u = -\nabla P + \Delta u 
& \text{on }]0,+\infty[\times\RR^3\\
\rule[1.3em]{0pt}{0pt}
\partial_ru_r + \frac{1}{r}u_r + \partial_zu_z = 0
& \text{on }]0,+\infty[\times\RR^3\\
\rule[1.3em]{0pt}{0pt}
u_r(0,r,\theta,z) = u_0(r,z) \cdot e_r
& \forall(r,z)\in\Omega, \forall\theta\in[0,2\pi[\\ 
\rule[1.3em]{0pt}{0pt}
u_\theta(0,r,\theta,z) = 0
& \forall(r,z)\in\Omega, \forall\theta\in[0,2\pi[\\ 
\rule[1.3em]{0pt}{0pt}
u_z(0,r,\theta,z) = u_0(r,z) \cdot e_z
& \forall(r,z)\in\Omega, \forall\theta\in[0,2\pi[
\end{array}\right.
\end{equation}
subject to $u_0\cdot e_\theta = 0$, which ensures that u remains axisymmetric with zero swirl at all times, as explained in \cite{MajBert}. Let us introduce the vector fields 
\begin{equation}\label{defuG1}
u^{G_1} = -\frac{1}{r}\partial_z\varphi\ e_r + \frac{1}{r}\partial_r\varphi\ e_z
\qquad \text{ and }\qquad u^{G_2} = -\partial_zu^{G_1}
\end{equation}
where $\varphi$ is the Stokes stream function of $G_1:(r,z)\mapsto\frac{r}{16\sqrt{\pi}}\exp(-\frac{r^2+z^2}{4})$,
\begin{equation}\label{defstreamG}
\varphi(r,z) = -\frac{r}{4\sqrt{\pi}}\ \partial_r\!\left( \frac{1}{\sqrt{r^2+z^2}}\int_0^{\sqrt{r^2+z^2}}e^{\mbox{$-\frac{\sigma^2}{4}$}}d\sigma \right),
\end{equation}
and let us adapt notations \eqref{defIJ} for the velocity: $I_0 = \frac{1}{2\pi}\int_{\RR^3}r\,\rot(u_0(r,z))\,rdrd\theta dz$ and $J_\infty$ is the finite limit of $J(t) = \frac{1}{2\pi}\int_{\RR^3}rz\,\rot(u(t,r,z))\,rdrd\theta dz$ when $t\to+\infty$. The study carried out through this paper leads then to the following theorem, giving information on the velocity's decay rate in $L^p(\RR^3)^3$ for $1<p\leqslant\infty$.

\begin{theorem}\label{theoremintrou}
Consider $u$ the solution of \eqref{NSaxy} with initial data $u_0$ such that $(r,z)\mapsto\frac{1}{r}\rot(u_0(r,z))$ is in $L^1(\RR^3)$. If $\|r\,\rot(u_0)\|_{L^1(\RR^3)}<\infty$ then
\begin{equation}\label{theoremintrour2}
u(t,r,z) = I_0u^{G_1}\left(\frac{r}{\sqrt{t}},\frac{z}{\sqrt{t}}\right)t^{-3/2} + \underset{t\to+\infty}{o}\left(t^{\frac{3}{2p}-\frac{3}{2}}\right)
\end{equation}
holds in $L^p(\RR^3)^3$ for every $p\in\ ]1,+\infty]$.\\
If moreover $\|(r^2+r|z|+|z|)\,\rot(u_0)\|_{L^1(\RR^3)}<\infty$, then 
\begin{equation}\label{theoremintrour3}
u(t,r,z) = I_0u^{G_1}\left(\frac{r}{\sqrt{t}},\frac{z}{\sqrt{t}}\right)t^{-3/2} + J_\infty u^{G_2}\left(\frac{r}{\sqrt{t}},\frac{z}{\sqrt{t}}\right)t^{-2} + \underset{t\to+\infty}{o}\left(t^{\frac{3}{2p}-2}\right)
\end{equation}
holds in $L^p(\RR^3)^3$ for every $p\in\ ]1,+\infty]$.
\end{theorem}

This theorem gives for example another insight into the paper of Miyakawa and Schonbek \cite{MiyakawaSchonbek2001} in the case of axisymmetry without swirl, namely the fact that, taking $p=2$, the solutions to the Navier-Stokes system decreasing faster than $t\mapsto t^{-5/4}$ are exactly those such that $I_0 = J_\infty = 0$. Naming 
\[
\mathscr{E} = \left\{\frac{1}{r}\rot(u_0)\in L^1(\RR^3) \,\middle|\, \|(r^2+r|z|+|z|)\,\rot(u_0)\|_{L^1(\RR^3)}<\infty, \,I_0 = J_\infty = 0\right\}
\]
the set of all such zero-swirl axisymmetric solutions (to be exact, the set of all initial vorticity data corresponding to such solutions), one can show by following the developments in \cite{GalWayR3} that the space $\mathscr{E}$ is a smooth invariant submanifold of codimension 2 of the weighted space $\{w\in L^1(\RR^3) \mid \|(\frac{1}{r}+r^2+r|z|+|z|)w\|_{L^1(\RR^3)} < \infty\}$.

This paper is written as follows. The second section analyses the semigroup associated with the linearised vorticity equation and its behaviour at large times. We give step by step the first two terms of its expansion, together with new $L^p-L^q$ estimates on the successive remainders. Considering the integral version of \eqref{voraxi}, these technical estimates will be fundamental to deal with both the linear and the non-linear components of the vorticity. The third section initiates the asymptotic study of the vorticity until its second order, following quite the same steps as for the semigroup. Adding increasingly strong hypotheses on the localisation of the initial data $\wo$, we refine the bound on the decay rate of $\w$. The final section discusses the consequences of this study on different levels. After speaking briefly about the emergence of the first resonant term, we state some corollary on the velocity's decay rate and prove theorem \ref{theoremintrou}. We then write the first two asymptotics when the swirl is nonzero. In the appendices, we develop some argument proving the local existence of the moments of $\w$ then we establish a specific Grönwall lemma, both of them will be used several times throughout section \ref{SectionAsymptoticExpansion}.

\subsection*{Notations}

\paragraph{}
$B$ will be reserved in this article to denote the Euler Beta function. Remember that, in particular, $B(1/2,1/2) = \pi$. For more details upon the properties of the $B$ function, see for example \cite{Olver}.

$C$ will denote any generic positive constant.
Constants named $C_\alpha$ or $C_{\wo}$ will be constants depending in their contexts on an $\alpha$ parameter or on the initial data $\wo$ of the vorticity equation \eqref{voraxi}.

$L^p(\Omega)$ will denote the Lebesgue space $\LLL{p}$ with respect to the two dimensional Lebesgue measure on $\Omega = \RR_+^*\times\RR$. 

$L^p(\RR^3)$, on the contrary, will refer to the usual $L^p(\RR^3,rdrd\theta dz)$ space with the three-dimensional Lebesgue measure.

$\dii$ will be the pseudo-divergence operator acting on a function $w = w_re_r + w_ze_z$ (in cylindrical coordinates) as $\dii w = \partial_rw_r + \partial_zw_z$.

\section{The Semigroup associated with the linearised equation}
\label{SectionSemigroup}

Before anything, some crucial results are useful to have in mind. Focus on \eqref{voraxi} and remark that the equation can also be expressed as 
\begin{equation}\label{vor}
\left\{\begin{array}{ll} \displaystyle
\partial_t\w + \dii(u\,\w) = \Delta\w - \frac{\w}{r^2}
& \text{on }]0,+\infty[\times\Omega
\\ 
\w(t,0,z) = 0 & \forall(t,0,z)\in[0,+\infty[\times\partial\Omega
\\ \rule[1.3em]{0pt}{0pt}
\w(0,r,z) = \wo(r,z) & \forall(r,z)\in\Omega
\end{array}\right.
\end{equation}
by writing $\dii(u\,\w) = \partial_r(u_r\w) + \partial_z(u_z\w)$ and using that $\di\,u = \partial_ru_r + \frac{1}{r}u_r + \partial_zu_z = 0$. The differential operator $\Delta - \frac{1}{r^2}$ is the generator of a semigroup $(S(t))_{t\geqslant0}$ of bounded linear operators on $\LLL{p}$ which is strongly continuous when $1\leqslant p<\infty$ (the calculation can be found in \cite{Gallay1}), defined by $S(0) = \id$ and 
\begin{equation}\label{defSt}
S(t)[\wo] : (r,z) \longmapsto \frac{r}{4\pi t^{5/2}}\int_\Omega K\left(\frac{r\rho}{t}\right)e^{\mbox{$-\frac{(r-\rho)^2 + (z-\zeta)^2}{4t}$}} \rho^2\wo(\rho,\zeta)\,d\rho d\zeta
\end{equation}
for any $t>0$, where
\begin{equation}\label{defK}
K(\tau) = \frac{1}{\tau\sqrt{4\pi}}\int_{-\pi}^\pi e^{\mbox{$-\tau\sin^2(\frac{\phi}{2})$}}\cos(\phi)d\phi.
\end{equation}
The properties of the function $K$ will be discussed in the next subsection.
The Duhamel formula says then that \eqref{vor} is equivalent to the integral equation
\begin{equation}\label{duhamel}
\w(t) = S(t)\wo - \int_0^tS(t-s)\left[\dii(u(s)\,\w(s))\right]ds
\qquad t\geqslant0,
\end{equation}
to be still considered on the open half-plane $\Omega = \{(r,z)\in\RR^2\mid r>0\}$ and with the Dirichlet boundary condition $\w = 0$ at $r=0$. Theorem 1.1 from \cite{Gallay1} tells that for an initial data $\wo$ in $\LL{1}$, this equation admits one unique global solution $\w\in \mathscr{C}^0([0,+\infty[,\LL{1})\cap\mathscr{C}^0(]0,+\infty[,\LL{\infty})$ which is infinitely differentiable on $]0,+\infty[\times\Omega$. Because the solution of \eqref{duhamel} is unique, it also satisfies for every $t_0>0$ the relation
\begin{equation}\label{duhameluni}
\w(t) = S(t-t_0)\w(t_0) - \int_{t_0}^tS(t-s)\,\dii(u(s)\,\w(s))\,ds \qquad\forall t\geqslant t_0.
\end{equation}
Indeed, following for a duration $t-t_0$ the solution stemming from $\w(t_0)$, which exists, one necessarily gets back to the solution $\w$ at time $t$.

Equation \eqref{duhamel} consists of two very different terms to deal with, one linear and the other nonlinear, both of which involving the semigroup $(S(t))_{t\geqslant0}$. Having a good idea as to how the semigroup behaves at infinity is a key point to the asymptotic study of $\w$.

\subsection{The Semigroup kernel}

The function $K$ defined by \eqref{defK} has the following properties.

\begin{proposition}\label{rkK}
$K$ is infinitely differentiable on $[0,+\infty[$ and the functions $\tau \mapsto (1+\tau)^{3/2+i}K^{(i)}(\tau)$ are all in $L^\infty([0,+\infty[)$ for $i\in\{0,1,2,3\}$.
\end{proposition}

This point will be crucial for propositions \ref{propS}, \ref{propS1} and \ref{propS2}, which give different kinds of estimates on the semigroup $(S(t))_{t\geqslant0}$.

\begin{proof}
Let us rewrite $K$ as
\begin{equation}\label{defKBessel}
K(\tau) = \frac{\sqrt{\pi}}{\tau}\,e^{-\tau/2}I_1\left(\frac{\tau}{2}\right)
\end{equation}
where
\begin{equation}\label{local7}
I_n(\tau) = \frac{1}{\pi}\int_0^\pi e^{\tau\cos\phi} \cos(n\phi)d\phi
\end{equation}
denotes the modified Bessel function of order $n$, $n\in\NN$. The desired properties of $K$ stem from the properties of $I_n$ for $n\in\{1,2,3,4\}$.

To begin with, writing the exponential in \eqref{local7} as a series and exchanging the series with the integral shows that $I_1$ can be expressed as $I_1(\tau) = \sum_{k\geqslant0} \frac{a_k}{(2k+1)!}\tau^{2k+1}$ where $a_k = \frac{1}{2}\times\frac{3}{4}\times...\times\frac{2k+1}{2k+2}$. Hence $K$ is infinitely differentiable on $[0,+\infty[$.

Let us then look at the behaviour of $K$, $K'$, $K''$ and $K'''$ at infinity. The asymptotic expansions of the modified Bessel functions are given by 
\begin{equation}\label{local8}
I_n(\tau) = \frac{1}{\sqrt{2\pi\tau}}e^\tau\left(1 - \frac{(4n^2-1)}{8\tau} + \frac{(4n^2-1)(4n^2-9)}{128\tau^2} + O\left(\tau^{-3}\right)\right) \quad\text{as }\tau\to+\infty,
\end{equation}
see for example the book of Franck Olver \cite[chapter 7]{Olver}, and in particular $I_1(\tau) = \frac{1}{\sqrt{2\pi\tau}}e^\tau(1+O(\tau^{-1}))$ so $K(\tau) \sim \tau^{-3/2}$ at infinity. Moreover, one gets by differentiating \eqref{local7} that $I_n'(\tau) = I_{n+1}(\tau) + \frac{n}{\tau}I_n(\tau)$ so successively 
\[\begin{array}{rl}
I_1'(\tau) &\hskip-0.6em = I_2(\tau) + \frac{1}{\tau}I_1(\tau), \\
\rule[1.3em]{0pt}{0pt} I_1''(\tau) &\hskip-0.6em = I_3(\tau) + \frac{3}{\tau}I_2(\tau), \\
\rule[1.3em]{0pt}{0pt} \text{ and }I_1'''(\tau) &\hskip-0.6em = I_4(\tau) + \frac{6}{\tau}I_3(\tau) + \frac{3}{\tau^2}I_2(\tau).
\end{array}\]
Calculating the derivatives of $K$ from \eqref{defKBessel} involves these three last formulae, which combined with \eqref{local8} lead finally to $K^{(i)}(\tau) = O(\tau^{-3/2-i})$ as $\tau\to+\infty$ for $i\in\{1,2,3\}$.
\end{proof}

\subsection{Estimates on the semigroup}\label{sectionMajoration}

Let us write some $\LL{p}-\LL{q}$ estimates on the moments of the semigroup $(S(t))_{t\geqslant0}$. Some of them already exist in \cite{Gallay1}, and can be stated the following way.

\begin{proposition}\label{prop4647}
Choose $1\leqslant p\leqslant q \leqslant \infty$ and $-1\leqslant\alpha\leqslant\beta\leqslant2$.
\begin{itemize}
\item[(i)] There exists a positive constant $C$ such that for all time $t>0$
\begin{equation}\label{46}
\|r^\alpha S(t)\wo\|_{\LL{q}} 
\leqslant C\|r^\beta\wo\|_{\LL{p}} 
\times t^{\frac{\alpha-\beta}{2}+\frac{1}{q}-\frac{1}{p}}
\end{equation}
provided that $(r,z)\mapsto r^\beta\wo(r,z)$ belongs to $\LL{p}$.
\item[(ii)] If $-1\leqslant\alpha\leqslant\beta\leqslant1$,
then there exists a positive constant $C$ such that for all time $t>0$
\begin{equation}\label{47}
\|r^\alpha  S(t)\dii w\|_{\LL{q}} 
\leqslant C \|r^\beta w\|_{\LL{p}} 
\times t^{-\frac{1}{2}+\frac{\alpha-\beta}{2}+\frac{1}{q}-\frac{1}{p}}
\end{equation}
provided that $(r,z)\mapsto r^\beta w(r,z)$ belongs to $(\LL{p})^2$.
\end{itemize}
\end{proposition}

The constants in \eqref{46} and \eqref{47} are actually independent of $p$, $q$, $\alpha$, $\beta$ and $\wo$. The constraints $-1\leqslant\alpha\leqslant\beta\leqslant2$ and $-1\leqslant\alpha\leqslant\beta\leqslant1$ suggested respectively for \eqref{46} and \eqref{47} are optimal as long as one wishes to get time-homogeneous bounds. Proposition \ref{prop4647} is sufficient to deal with the existence, uniqueness and regularity to \eqref{vor}. Proposition \ref{prop4647} is sufficient as well to give the decay rates of the vorticity when the initial data $\wo$ is such that $r\wo$ is in $\LL{1}$, using inequality \eqref{47} up to its limit $\beta=1$ (this will be the object of lemma \ref{lemmamajorationsr}). However, in order to go further into the asymptotic expansion of $\w$, we shall need estimates on the $r$-moments of $S(t)$ higher than 2 as well as some of its $r$-and-$z$-moments that proposition \ref{prop4647} does not handle.

In practice, the following inequalities will be required in the next section; some concerning the semigroup strictly speaking, like in \eqref{46},
\begin{align}
\|r^\alpha z\,S(t)\wo\|_{\LL{q}} 
& \leqslant C \|r^\alpha(r+|z|)\,\wo\|_{\LL{p}} 
\times t^{\frac{1}{q}-\frac{1}{p}},
\label{propSralphaz}
\\
\|r^{2+\alpha}S(t)\wo\|_{\LL{q}} \el{1.3em}
& \leqslant C \|r^2(t^{\frac{\alpha}{2}}+r^\alpha)\,\wo\|_{\LL{p}} 
\times t^{\frac{1}{q}-\frac{1}{p}},
\label{propSr2alpha}
\\
\|r^\alpha z^2S(t)\wo\|_{\LL{q}} \el{1.3em}
& \leqslant C \|(r^2t^{\frac{\alpha}{2}}+r^\alpha z^2)\,\wo\|_{\LL{p}} 
\times t^{\frac{1}{q}-\frac{1}{p}},
\label{propSralphaz2}
\end{align}
and the others concerning the semigroup acting upon the $\dii$ operator, like in \eqref{47},
\begin{align}
\|r^2S(t)\dii w\|_{\LL{q}} 
& \leqslant C \|r(\sqrt{t}+r)w\|_{\LL{p}} 
\times t^{-\frac{1}{2}+\frac{1}{q}-\frac{1}{p}},
\label{propSdivr2}
\\
\|r^\alpha z\,S(t)\dii w\|_{\LL{q}} \el{1.3em}
& \leqslant C \|(r\,t^{\frac{\alpha}{2}}+r^\alpha|z|)\,w\|_{\LL{p}} 
\times t^{-\frac{1}{2}+\frac{1}{q}-\frac{1}{p}},
\label{propSdivralphaz}
\\
\|r^{2+\alpha}S(t)\dii w\|_{\LL{q}} \el{1.3em}
& \leqslant C \|(r\,t^{\frac{1+\alpha}{2}}+r^{2+\alpha})\,w\|_{\LL{p}} 
\times t^{-\frac{1}{2}+\frac{1}{q}-\frac{1}{p}},
\label{propSdivr2alpha}
\\
\|r^\alpha z^2S(t)\dii w\|_{\LL{q}} \el{1.3em}
& \leqslant C \|(r\,t^{\frac{1+\alpha}{2}}+r^\alpha z^2)\,w\|_{\LL{p}} 
\times t^{-\frac{1}{2}+\frac{1}{q}-\frac{1}{p}}
\label{propSdivralphaz2}
\end{align}
where $\alpha$ lies in $[0,1]$. These inequalities can be proven simultaneously as consequences of the following statement. 

\begin{proposition}\label{propS}
Choose $1\leqslant p\leqslant q \leqslant \infty$ and $\alpha',\gamma\geqslant 0$. Define for $t\in\ ]0,+\infty[$ and $(r,z)\in\Omega$
\begin{equation}\label{defN}
N(t,r,z) = \left(\left(\frac{r}{\sqrt{t}}\right)^\beta + \left(\frac{r}{\sqrt{t}}\right)^{\beta'}\left|\frac{z}{\sqrt{t}}\right|^{\gamma}\right)\,\left(1+\left(\frac{r}{\sqrt{t}}\right)^{\alpha'}\right).
\end{equation}
\begin{itemize}
\item[(i)] If $-1\leqslant\alpha\leqslant\beta,\beta'\leqslant2$, 
there exists a positive constant $C_{\alpha',\gamma}$ such that for all time $t>0$
\begin{equation}\label{propS46}
\|r^{\alpha+\alpha'}z^{\gamma}S(t)\wo\|_{\LL{q}} 
\leqslant C_{\alpha',\gamma} \left\|N(t)\wo\right\|_{\LL{p}} 
\times t^{\frac{\alpha+\alpha'+\gamma}{2}+\frac{1}{q}-\frac{1}{p}}
\end{equation}
provided that $(r,z)\mapsto(r^\beta+r^{\beta+\alpha'}+r^{\beta'}|z|^\gamma+r^{\beta'+\alpha'}|z|^\gamma)\wo(r,z)\in\LL{p}$.
\item[(ii)] If $-1\leqslant\alpha\leqslant\beta,\beta'\leqslant1$, 
there exists a positive constant $C_{\alpha',\gamma}$ such that for all time $t>0$
\begin{equation}\label{propS47}
\|r^{\alpha+\alpha'}z^{\gamma}S(t)\dii w\|_{\LL{q}} 
\leqslant C_{\alpha',\gamma} \left\|N(t)w\right\|_{\LL{p}} 
\times t^{-\frac{1}{2}+\frac{\alpha+\alpha'+\gamma}{2}+\frac{1}{q}-\frac{1}{p}}
\end{equation}
provided that $(r,z)\mapsto(r^\beta+r^{\beta+\alpha'}+r^{\beta'}|z|^\gamma+r^{\beta'+\alpha'}|z|^\gamma)w(r,z)\in(\LL{p})^2$.
\end{itemize}
\end{proposition}

If one wishes to know explicitly how the constant in \eqref{propS46} and \eqref{propS47} depends on $\alpha'$ and $\gamma$, they may write the optimal majorations
\begin{equation}\label{deltasum}
\forall\delta>0,\,\forall x\geqslant0\qquad (1+x)^\delta \leqslant \max(1,2^{\delta-1})\,(1+x^\delta)
\end{equation}
and 
\begin{equation}\label{deltaexp}
\forall\delta\geqslant0,\,\forall x\geqslant0\qquad x^\delta \leqslant \left(\frac{2\delta}{e}\right)^{\delta/2} e^{\frac{1}{4}x^2}
\end{equation}
and conclude while reading the proof coming below that $C_{\alpha',\gamma}$ can be taken as 
\[
C_{\alpha',\gamma} = C_0 \times \max(1,2^{\alpha'-1})\max(1,2^{\gamma-1}) \max\left(1,\left(\frac{2\alpha'}{e}\right)^{\alpha'/2}\right)\max\left(1,\left(\frac{2\gamma}{e}\right)^{\gamma/2}\right),
\]
where $C_0$ is determined by the bounds on $K$ and its derivatives (\textit{cf} proposition \ref{rkK}), by the value of $\|\exp(-\frac{1}{8}(r^2+z^2))\|_{\LL{1}} = 4\pi$, and by the maximum value of $\left(\frac{2\delta}{e}\right)^{\delta/2}$ for $\delta$ between 0 and 5.

\begin{proof}
Let us write from \eqref{defSt} that for all time $t>0$
\[
r^{\alpha+\alpha'}z^\gamma S(t)\wo(r,z) = \frac{1}{4\pi t^{5/2}}\int_\Omega r^{\alpha'}z^\gamma e^{-X}\ r^{\alpha+1}\rho^2K\left(\frac{r\rho}{t}\right)\ \wo(\rho,\zeta)\,d\rho d\zeta
\]
where we note $X = X(t,r-\rho,z-\zeta) = \frac{1}{4t}((r-\rho)^2 + (z-\zeta)^2)$ for the clarity of the demonstration. Using an integration by part, write as well that, for all time $t>0$,
\[
r^{\alpha+\alpha'}z^\gamma S(t)[\dii w](r,z) = -\frac{1}{4\pi t^{3}}\int_\Omega r^{\alpha'}z^\gamma e^{-X}\ r^\alpha(A_rw_r + A_zw_z)\,d\rho d\zeta
\]
where $w=(w_r,w_z)$ and the terms $A_r$ and $A_z$ have the following expressions:
\[
A_r(t,r,z,\rho,\zeta) = \frac{r\rho^2}{2\sqrt{t}}(r-\rho) K\left(\frac{r\rho}{t}\right) + \frac{r^2\rho^2}{2\sqrt{t}} K'\left(\frac{r\rho}{t}\right) + 2r\rho\sqrt{t} K\left(\frac{r\rho}{t}\right)
\]
and
\[
A_z(t,r,z,\rho,\zeta) = \frac{r\rho^2}{2\sqrt{t}}(z-\zeta) K\left(\frac{r\rho}{t}\right).
\]
In both cases, cut the factor $r^{\alpha'}z^\gamma e^{-X}$ into four parts according to \eqref{deltasum} then \eqref{deltaexp}:
\[
r^{\alpha'}|z|^\gamma
\leqslant C(|r-\rho|^{\alpha'} + \rho^{\alpha'})(|z-\zeta|^\gamma + |\zeta|^\gamma)
\leqslant C 
(t^{\frac{\alpha'}{2}} + \rho^{\alpha'})(t^{\frac{\gamma}{2}} + |\zeta|^\gamma)e^{\frac{X}{4}}.
\]
Afterwards, thanks to proposition \ref{rkK} and \eqref{deltaexp}, let us show that the factor $r^{\alpha+1}\rho^2K(\frac{r\rho}{t})$ is less than or equal to the quantity $C (\frac{\rho}{\sqrt{t}})^\beta e^{\frac{X}{4}}t^{\frac{\alpha+3}{2}}$ as follows.
\begin{itemize}[label=$\circ$] 
\item If ($\alpha+\beta-1\geqslant0$ and $0\leqslant r\leqslant 2\rho$) or ($\alpha+\beta-1\leqslant0$ and $0\leqslant\rho<2r$) then
\[
r^{\alpha+1}\rho^2 K\left(\frac{r\rho}{t}\right)
= \frac{r^{\frac{\alpha+\beta-1}{2}}\rho^{\frac{\beta-\alpha+1}{2}}}{t^{\frac{\beta-\alpha-3}{2}}}\times \left(\frac{r\rho}{t}\right)^{\frac{3+\alpha-\beta}{2}}K\left(\frac{r\rho}{t}\right)
\leqslant C \left(\frac{\rho}{\sqrt{t}}\right)^\beta e^{\frac{X}{4}} t^{\frac{\alpha+3}{2}}
\]
since $0\leqslant 3+\alpha-\beta\leqslant3$ (\textit{cf} proposition \ref{rkK}) and applying \eqref{deltasum} with $\delta=0$.
\item If $\alpha+\beta-1\geqslant0$ and $0\leqslant2\rho\leqslant r$ then
\[
r^{\alpha+1}\rho^2 K\left(\frac{r\rho}{t}\right)
= \frac{r^{\alpha+\beta-1}\rho^\beta}{t^{\frac{2\beta-4}{2}}} \times \left(\frac{r\rho}{t}\right)^{2-\beta}K\left(\frac{r\rho}{t}\right)
\leqslant C \left(\frac{\rho}{\sqrt{t}}\right)^\beta e^{\frac{X}{4}} t^{\frac{\alpha+3}{2}}
\]
since $0\leqslant 2-\beta\leqslant\frac{3}{2}$ and $r\leqslant |r-\rho|+\rho \leqslant |r-\rho|+\frac{r}{2}$ so $r\leqslant 2|r-\rho| \leqslant 4\sqrt{t\,X}$.
\item If finally $\alpha+\beta-1\leqslant0$ and $0\leqslant 2r\leqslant\rho$ then
\[
r^{\alpha+1}\rho^2 K\left(\frac{r\rho}{t}\right)
= \frac{\rho^\beta \rho^{-(\alpha+\beta-1)}}{t^{\frac{-2\alpha-2}{2}}} \times \left(\frac{r\rho}{t}\right)^{\alpha+1}K\left(\frac{r\rho}{t}\right)
\leqslant C \left(\frac{\rho}{\sqrt{t}}\right)^\beta e^{\frac{X}{4}} t^{\frac{\alpha+3}{2}}
\]
since $0\leqslant \alpha+1\leqslant\frac{3}{2}$ and $\rho\leqslant |\rho-r|+r \leqslant |\rho-r|+\frac{\rho}{2}$ so $\rho\leqslant 2|\rho-r| \leqslant 4\sqrt{t\,X}$.
\end{itemize}
Taking $\beta'$ instead of $\beta$ gives that $r^{\alpha+1}\rho^2K(\frac{r\rho}{t}) \leqslant C (\frac{\rho}{\sqrt{t}})^{\beta'}e^{\frac{X}{4}}t^{\frac{\alpha+3}{2}}$ for $-1\leqslant\alpha\leqslant\beta'\leqslant2$ as well.

Let us treat the factors $r^\alpha |A_r|$ and $r^\alpha |A_z|$ in the same way when $-1\leqslant\alpha\leqslant\beta\leqslant1$, noticing that they are sums of terms of the form $r^{\alpha+1+i}\rho^{1+j}X^{k/2} t^{\frac{1-i-j}{2}} 
\,K^{(i)}(\frac{r\rho}{t})$ where $i,j,k$ are in $\{0,1\}$. 
\begin{itemize}[label=$\circ$] 
\item If ($\alpha+\beta+i-j\geqslant0$ and $0\leqslant r\leqslant 2\rho$) or ($\alpha+\beta+i-j\leqslant0$ and $0\leqslant\rho<2r$) then
\begin{align*}
\frac{r^{\alpha+1+i}\rho^{1+j}}{t^{\frac{i+j-1}{2}}}X^{\frac{k}{2}} K^{(i)}\left(\frac{r\rho}{t}\right)
& = \frac{r^{\frac{\alpha+\beta+i-j}{2}}\rho^{\frac{\beta-\alpha-i+j}{2}}}{t^{\frac{\beta-\alpha-3}{2}}}X^{\frac{k}{2}} \times \left(\frac{r\rho}{t}\right)^{\frac{2+\alpha+i+j-\beta}{2}}K^{(i)}\left(\frac{r\rho}{t}\right)\\
&\el{2.2em} \leqslant C \left(\frac{\rho}{\sqrt{t}}\right)^\beta e^{\frac{X}{4}} t^{\frac{\alpha+3}{2}}
\end{align*}
since $0\leqslant 2+\alpha+i+j-\beta\leqslant3+2i$.
\item If $\alpha+\beta+i-j\geqslant0$ and $0\leqslant2\rho\leqslant r$ then
\begin{align*}
\frac{r^{\alpha+1+i}\rho^{1+j}}{t^{\frac{i+j-1}{2}}}X^{\frac{k}{2}} K^{(i)}\left(\frac{r\rho}{t}\right)
& = \frac{r^{\alpha+\beta+i-j}\rho^\beta}{t^{\frac{2\beta+i-j-3}{2}}}X^{\frac{k}{2}} \times \left(\frac{r\rho}{t}\right)^{1+j-\beta}K^{(i)}\left(\frac{r\rho}{t}\right)\\
&\el{2.2em} \leqslant C \left(\frac{\rho}{\sqrt{t}}\right)^\beta e^{\frac{X}{4}} t^{\frac{\alpha+3}{2}}
\end{align*}
since $0\leqslant 1+j-\beta\leqslant\frac{3+2i}{2}$ and $r\leqslant |r-\rho|+\rho \leqslant |r-\rho|+\frac{r}{2}$ so $r\leqslant 2|r-\rho| \leqslant 4\sqrt{t\,X}$.
\item If finally $\alpha+\beta+i-j\leqslant0$ and $0\leqslant 2r\leqslant\rho$ then
\begin{align*}
\frac{r^{\alpha+1+i}\rho^{1+j}}{t^{\frac{i+j-1}{2}}}X^{\frac{k}{2}} K^{(i)}\left(\frac{r\rho}{t}\right)
& = \frac{\rho^\beta \rho^{-(\alpha+\beta+i-j)}}{t^{\frac{-2\alpha-i+j-3}{2}}}X^{\frac{k}{2}} \times \left(\frac{r\rho}{t}\right)^{\alpha+i+1}K^{(i)}\left(\frac{r\rho}{t}\right)\\
&\el{2.2em} \leqslant C \left(\frac{\rho}{\sqrt{t}}\right)^\beta e^{\frac{X}{4}} t^{\frac{\alpha+3}{2}}
\end{align*}
since $0\leqslant \alpha+i+1\leqslant\frac{3+2i}{2}$ and $\rho\leqslant |\rho-r|+r \leqslant |\rho-r|+\frac{\rho}{2}$ so $\rho \leqslant 2|\rho-r| \leqslant 4\sqrt{t\,X}$.
\end{itemize}
Let us take $\beta'$ instead of $\beta$ and get that $r^\alpha|A_r| + r^\alpha|A_z| \leqslant C (\frac{\rho}{\sqrt{t}})^{\beta'}e^{\frac{X}{4}}t^{\frac{\alpha+3}{2}}$ as well.

We have now shown that on the one hand
\[
|r^{\alpha+\alpha'}z^\gamma S(t)\wo(r,z)| \leqslant \frac{C_{\alpha',\gamma}}{4\pi}\int_\Omega e^{-X/2}\ N(t,\rho,\zeta)\,|\wo(\rho,\zeta)|\,d\rho d\zeta \times t^{\frac{\alpha+\alpha'+\gamma}{2}-1}
\]
and on the other hand
\[
|r^{\alpha+\alpha'}z^\gamma S(t)\dii w(r,z)| \leqslant \frac{C_{\alpha',\gamma}}{4\pi}\int_\Omega e^{-X/2}\ N(t,\rho,\zeta)\,|w(\rho,\zeta)|\,d\rho d\zeta \times t^{-\frac{1}{2}+\frac{\alpha+\alpha'+\gamma}{2}-1},
\]
which concludes the proof thanks to Young's inequality for convolution.
\end{proof}

Let us remark eventually that all of the estimates \eqref{propSralphaz} to \eqref{propSdivralphaz2} are indeed obtained from proposition \ref{propS} with well choosen values of $\alpha$, $\alpha'$, $\beta$, $\beta'$ and $\gamma$. Incidentally, let us see that proposition \ref{prop4647} is also contained in proposition \ref{propS}, corresponding to the cases when $\beta'=\beta$ and $\alpha'=\gamma=0$.

\subsection{Approximation of the semigroup at first order}
\label{sectionApproximationFirstOrder}

The goal is now to specify the behaviour of the semigroup $(S(t))_{t\geqslant0}$ at large times. We know from \eqref{46} that $\|S(t)\wo\|_{\LL{1}} \leqslant C\woo{r^\beta}t^{-\beta/2}$ as long as $\beta\in[0,2]$. In particular, when $(r,z)\mapsto r^2\wo(r,z)$ is in $\LL{1}$, it appears that the equivalent
\begin{equation}\label{equivalentS1L1}
S(t)\wo \ \underset{t\to+\infty}{\sim} \ \frac{I_0r}{16\sqrt{\pi}}e^{\mbox{$-\frac{r^2+z^2}{4t}$}}\, t^{-5/2}
\end{equation}
holds whenever $I_0 = \int_\Omega r^2\wo(r,z)\,drdz \neq 0$.  This will be the object of lemma \ref{lemmaconvergenceS}. Hence, as long as $I_0$ is nonzero, $t\mapsto S(t)\wo$ decreases in $\LL{1}$ exactly as $t\mapsto t^{-1}$ up to a multiplicative constant. One thing to notice here is the preponderant role of the radial coordinate in axisymmetry, for until now all the hypotheses made relate to the $r$-moments of $\wo$ and no additional spatial decay has been asked so far along the vertical axis.

Let us introduce the Gaussian-like function
\begin{equation}\label{defG1}
G_1(r,z) = \frac{r}{16\sqrt{\pi}}e^{\mbox{$-\frac{r^2+z^2}{4}$}}
\end{equation}
on $\Omega$ and the notation
\begin{equation}\label{defS1}
S_1(t)\wo : (r,z) \mapsto S(t)\wo(r,z) - \left(\int_\Omega \rho^2\wo(\rho,\zeta)\,d\rho d\zeta\right) G_1\left(\frac{r}{\sqrt{t}},\frac{z}{\sqrt{t}}\right)t^{-2}
\end{equation}
when $\woo{(1+r^2)}<\infty$ and $t>0$. Despite its denomination, chosen because $S_1$ gives the first order approximation of $S$ at large times, the family of operators $(S_1(t))_{t>0}$ is not a semigroup and is not strongly continuous at $t=0$. We now state some estimates on $(S_1(t))_{t>0}$ comparable to those from proposition \ref{propS}.

\begin{proposition}\label{propS1}
Choose $1\leqslant p\leqslant q \leqslant \infty$ such that $p<2$ and $-1\leqslant\alpha\leqslant \beta \leqslant 2$.
\begin{itemize}
\item[(i)] If $3-\frac{2}{p}<\beta\leqslant2$
then there exists a positive constant $C$ such that for all time $t>0$
\begin{equation}\label{propS1r}
\|r^\alpha S_1(t)\wo\|_{\LL{q}} 
\leqslant \frac{C}{\frac{1}{p}+\frac{\beta+1}{2}-2}\, \|r^\beta(r+|z|)\,\wo\|_{\LL{p}} 
\times t^{-\frac{1}{2}+\frac{\alpha-\beta}{2}+\frac{1}{q}-\frac{1}{p}}
\end{equation}
provided that $(r,z)\mapsto r^2\wo(r,z)\in\LL{1}$ and $(r,z)\mapsto r^\beta(r+|z|)\wo(r,z)\in\LL{p}$.
\item[(ii)] If $2-\frac{2}{p}<\beta\leqslant1$
then there exists a positive constant $C$ such that for all time $t>0$
\begin{equation}\label{propS1divr}
\|r^\alpha S_1(t)\dii w\|_{\LL{q}} 
\leqslant \frac{C}{\frac{1}{p}+\frac{\beta+2}{2}-2}\, \|r^\beta(r+|z|)\,w\|_{\LL{p}} 
\times t^{-1+\frac{\alpha-\beta}{2}+\frac{1}{q}-\frac{1}{p}}
\end{equation}
provided that $(r,z)\mapsto r^2\dii w(r,z)\in\LL{1}$ and $(r,z)\mapsto r^\beta(r+|z|)w(r,z)\in(\LL{p})^2$.
\end{itemize}
\end{proposition}

The constants in \eqref{propS1r} and \eqref{propS1divr} can be shown to be independent of $p$, $q$, $\alpha$, $\beta$ and $\wo$. The proof follows the same procedure than the one of proposition \ref{propS}, with the difference that we shall conduct this one in the self-similar space-and-time coordinates $(r\sqrt{t},z\sqrt{t})$.

\begin{proof}
Let us write for all $t>0$ and $(r,z),(\rho,\zeta)\in\Omega$ the first order asymptotic expansion with integral remainder
\begin{equation}\label{demopropS1Taylor}
K\left(\frac{r\rho}{\sqrt{t}}\right) e^{-X(t)}
= \frac{\sqrt{\pi}}{4}e^{-\frac{1}{4}(r^2 + z^2)}
+ \int_t^{+\infty} \frac{1}{2s^{3/2}} A_0\,e^{-X(s)}ds
\end{equation}
where we shall use the notation $X(t) 
= \frac{1}{4}((r-\frac{\rho}{\sqrt{t}})^2 + (z-\frac{\zeta}{\sqrt{t}})^2)$ and where $A_0$ denotes the quantity
\begin{equation}\label{demopropS1A0}
A_0 = r\rho K'\left(\frac{r\rho}{\sqrt{s}}\right) + \frac{1}{2}\left(\rho\left(r-\frac{\rho}{\sqrt{s}}\right) + \zeta\left(z-\frac{\zeta}{\sqrt{s}}\right)\right) K\left(\frac{r\rho}{\sqrt{s}}\right).
\end{equation}
From that expansion and the expression of the semigroup \eqref{defSt} in the self-similar coordinates $(r\sqrt{t},z\sqrt{t})$, one gets that on the one hand
\begin{equation}\label{demopropS1S1}
r^\alpha S_1(t)\wo(r\sqrt{t},z\sqrt{t}) = \frac{1}{8\pi t^2}\int_\Omega\int_t^{+\infty}\frac{1}{s^{3/2}}\ r^{\alpha+1}\rho^2A_0\,e^{-X(s)}\ \wo(\rho,\zeta)\,ds\,d\rho d\zeta
\end{equation}
for all time $t>0$, and on the other hand after integrating by parts
\begin{equation}\label{demopropS1S1div}
r^\alpha S_1(t)[\dii w](r\sqrt{t},z\sqrt{t}) = -\frac{1}{8\pi t^{2}}\int_\Omega\int_t^{+\infty}\frac{1}{s^2}\ r^\alpha(A_rw_r + A_zw_z) e^{-X(s)}ds\,d\rho d\zeta
\end{equation}
where $w=(w_r,w_z)$ and
\begin{equation}\label{demopropS1ArAz}
\begin{array}{l}
A_r 
= \left(\frac{r\rho^2}{2\sqrt{s}} \left(r-\frac{\rho}{\sqrt{s}}\right)A_0 + \partial_\rho(r\rho^2A_0)\right)\sqrt{s}
\\ \el{1.7em}
A_z 
= \left(\frac{r\rho^2}{2\sqrt{s}} \left(z-\frac{\zeta}{\sqrt{s}}\right)A_0 + \partial_\zeta(r\rho^2A_0)\right)\sqrt{s}.
\end{array}
\end{equation}
The rest of the demonstration uses the arguments already encountered in the proof of proposition \ref{propS}. The calculation gives that $r\rho^2A_0$, $A_r$ and $A_z$ are sums of terms of the form
\[
\frac{r^{1+i}\rho^{1+j}}{s^{\frac{j-1}{2}}} \rho^k\left(r-\frac{\rho}{\sqrt{s}}\right)^{k'}\! \zeta^{1-k}\left(z-\frac{\zeta}{\sqrt{s}}\right)^{k''}\! K^{(i)}\left(\frac{r\rho}{\sqrt{s}}\right)
\]
with $j,k\in\{0,1\}$ and $i,k',k''\in\{0,1,2\}$. It is important to note here that the terms with $j=0$ appear only in $A_r$ and $A_z$, when $\beta\leqslant1$. Then see that according to \eqref{deltaexp}
\[
\rho^k\left|r-\frac{\rho}{\sqrt{s}}\right|^{k'} |\zeta|^{1-k}\left|z-\frac{\zeta}{\sqrt{s}}\right|^{k''} \leqslant (\rho + |\zeta|)\left|r-\frac{\rho}{\sqrt{s}}\right|^{k'} \left|z-\frac{\zeta}{\sqrt{s}}\right|^{k''} \leqslant C(\rho + |\zeta|)e^{\frac{X(s)}{4}}
\]
for every $k\in\{0,1\}$ and $k',k''\in\{0,1,2\}$.

Afterwards, thanks to proposition \ref{rkK} and to \eqref{deltaexp}, let us show that the factors of the form $r^\alpha\times r^{1+i}\rho^{1+j}s^{\frac{1-j}{2}} K^{(i)}(\frac{r\rho}{\sqrt{s}})$ are less than or equal to the quantity $C \rho^\beta s^{\frac{2-\beta}{2}} e^{\frac{X(s)}{4}}$ as follows. For $i\in\{0,1,2\}$ and $j\in\{0,1\}$ corresponding to the factors of interest, meaning that $j=0$ only when $\beta\leqslant 1$,
\begin{itemize}[label=$\circ$] 
\item if ($\alpha+\beta+i-j\geqslant0$ and $0\leqslant r\leqslant 2\frac{\rho}{\sqrt{s}}$) or ($\alpha+\beta+i-j\leqslant0$ and $0\leqslant\frac{\rho}{\sqrt{s}}<2r$) then
\begin{align*}
\frac{r^{\alpha+1+i}\rho^{1+j}}{s^{\frac{j-1}{2}}} K^{(i)}\left(\frac{r\rho}{\sqrt{s}}\right)
& = \frac{r^{\frac{\alpha+\beta+i-j}{2}}\rho^{\frac{\beta-\alpha-i+j}{2}}}{s^{\frac{\beta-\alpha-i+j-4}{4}}} \times \left(\frac{r\rho}{\sqrt{s}}\right)^{\frac{2+\alpha+i+j-\beta}{2}}K^{(i)}\left(\frac{r\rho}{\sqrt{s}}\right)\\
&\el{1.7em} \leqslant C \rho^\beta s^{\frac{2-\beta}{2}} e^{\frac{X(s)}{4}}
\end{align*}
since $0\leqslant 2+\alpha+i+j-\beta\leqslant3+2i$,
\item if $\alpha+\beta+i-j\geqslant0$ and $0\leqslant2\frac{\rho}{\sqrt{s}}\leqslant r$ then
\begin{align*}
\frac{r^{\alpha+1+i}\rho^{1+j}}{s^{\frac{j-1}{2}}} K^{(i)}\left(\frac{r\rho}{\sqrt{s}}\right)
& = \frac{r^{\alpha+\beta+i-j}\rho^\beta}{s^{\frac{\beta-2}{2}}} \times \left(\frac{r\rho}{\sqrt{s}}\right)^{1+j-\beta}K^{(i)}\left(\frac{r\rho}{\sqrt{s}}\right)\\
&\el{1.7em} \leqslant C \rho^\beta s^{\frac{2-\beta}{2}} e^{\frac{X(s)}{4}}
\end{align*}
since $0\leqslant 1+j-\beta\leqslant\frac{3+2i}{2}$ and $r\leqslant |r-\frac{\rho}{\sqrt{s}}|+\frac{\rho}{\sqrt{s}} \leqslant |r-\frac{\rho}{\sqrt{s}}|+\frac{r}{2}$ so $r\leqslant 2|r-\frac{\rho}{\sqrt{s}}| \leqslant 4\sqrt{X(s)}$,
\item and if $\alpha+\beta+i-j\leqslant0$ and $0\leqslant 2r\leqslant\frac{\rho}{\sqrt{s}}$ then
\begin{align*}
\frac{r^{\alpha+1+i}\rho^{1+j}}{s^{\frac{j-1}{2}}} K^{(i)}\left(\frac{r\rho}{\sqrt{s}}\right)
& = \frac{\rho^\beta}{s^{\frac{\beta-2}{2}}}\left(\frac{\rho}{\sqrt{s}}\right)^{-(\alpha+\beta+i-j)} \times \left(\frac{r\rho}{\sqrt{s}}\right)^{\alpha+i+1}K^{(i)}\left(\frac{r\rho}{\sqrt{s}}\right)\\
&\el{1.7em} \leqslant C \rho^\beta s^{\frac{2-\beta}{2}} e^{\frac{X(s)}{4}}
\end{align*}
since $0\leqslant \alpha+i+1\leqslant\frac{3+2i}{2}$ and $\frac{\rho}{\sqrt{s}}\leqslant |\frac{\rho}{\sqrt{s}}-r|+r \leqslant |\frac{\rho}{\sqrt{s}}-r|+\frac{\rho}{2\sqrt{s}}$ so $\frac{\rho}{\sqrt{s}} \leqslant 2|\frac{\rho}{\sqrt{s}}-r| \leqslant 4\sqrt{X(s)}$.
\end{itemize}

It is only left to write, using first a Fubini inversion between the time and space integrals and then Young's inequality for convolution, that on the one hand
\begin{align*}
\|r^\alpha S_1(t)\wo(r\sqrt{t},z\sqrt{t})\|_{\LL{q}} 
& \leqslant \left\| \frac{C}{t^2}\int_\Omega\int_t^{+\infty} s^{-\frac{\beta+1}{2}} e^{-\frac{X(s)}{2}}\ \rho^\beta(\rho + |\zeta|)\,|\wo(\rho,\zeta)|\,ds\,d\rho d\zeta \right\|_{\LL{q}} \\
& \leqslant \frac{C}{t^2} \int_t^{+\infty}s^{-\frac{\beta+1}{2}} \left\| \int_\Omega e^{-\frac{X(s)}{2}}\ \rho^\beta(\rho + |\zeta|)\,|\wo(\rho,\zeta)|\,d\rho d\zeta \right\|_{\LL{q}}\,ds \\
& \leqslant \frac{C}{t^2} \int_t^{+\infty}s^{-\frac{\beta+1}{2}+1-\frac{1}{p}}ds\ \left\|r^\beta(r + |z|)\wo(r,z)\right\|_{\LL{p}} 
\end{align*}
and on the other hand
\begin{align*}
\|r^\alpha S_1(t)\dii w(r\sqrt{t},z\sqrt{t})\|_{\LL{q}} 
& \leqslant \left\| \frac{C}{t^2}\int_\Omega\int_t^{+\infty} s^{-\frac{\beta+2}{2}} e^{-\frac{X(s)}{2}}\ \rho^\beta(\rho + |\zeta|)\,|w(\rho,\zeta)|\,ds\,d\rho d\zeta \right\|_{\LL{q}} \\
& 
\leqslant \frac{C}{t^2} \int_t^{+\infty}s^{-\frac{\beta+2}{2}} \left\| \int_\Omega e^{-\frac{X(s)}{2}}\ \rho^\beta(\rho + |\zeta|)\,|w(\rho,\zeta)|\,d\rho d\zeta \right\|_{\LL{q}}\,ds \\
& \leqslant \frac{C}{t^2} \int_t^{+\infty}s^{-\frac{\beta+2}{2}+1-\frac{1}{p}}ds\ \left\|r^\beta(r + |z|)w(r,z)\right\|_{\LL{p}},
\end{align*}
and proof is made with a last change of variables on the left hand sides of these two last inequalities.
\end{proof}

We can now prove the statement made in \eqref{equivalentS1L1}, which is true in every $\LL{p}$ for $1\leqslant p\leqslant\infty$. The demonstration will be performed again in the self-similar coordinates $(z\sqrt{t},z\sqrt{t})$.

\begin{lemma}\label{lemmaconvergenceS}
Take $\wo\in\LL{1}$ such that $(r,z)\mapsto r^2\wo(r,z)\in\LL{1}$. Remember that $(S(t))_{t\geqslant0}$ is the semigoup defined by \eqref{defSt} and $I_0 = \int_\Omega r^2\wo(r,z)\,drdz$. Then
\begin{equation}\label{lemmaconvergenceSconvergence}
S(t)\wo(r,z)\ = \ \frac{I_0r}{16\sqrt{\pi}}e^{\mbox{$-\frac{r^2+z^2}{4t}$}}\ t^{-5/2} + \underset{t\to+\infty}{o}\left(t^{\frac{1}{p}-2}\right)
\end{equation}
in $\LL{p}$ for every $p\in[1,+\infty]$.
\end{lemma}

\begin{proof}[Proof]
Let us recall that $G_1 : \Omega \to\RR : (r,z) \mapsto \frac{r}{16\sqrt{\pi}}\exp(-\frac{r^2+z^2}{4})$ and $S_1:\wo\mapsto S(t)\wo - I_0G_1(\frac{\cdot}{\sqrt{t}})t^{-2}$. The goal is to show that $\|t^2S_1(t)\wo(\cdot\sqrt{t})\|_{\LL{p}}$ tends to zero as $t$ goes to infinity for $p$ in $[1,+\infty]$.

Let us consider for any $a>1$ the indicator function $\chi_a$ of the truncated subset $[\frac{1}{a},a]\times[-a,a]$ of $\Omega$, and the quantity $I_0^a = \int_{1/a}^a\int_{-a}^a \rho^2\wo(\rho,\zeta)\,d\rho d\zeta = \int_\Omega\rho^2\chi_a(\rho,\zeta)\wo(\rho,\zeta)\,d\rho d\zeta$ which tends towards $I_0$ as $a$ goes to infinity.
Thanks to estimate \eqref{propS1r} with $\alpha=0$ and $\beta=2$, we see that
\[
\begin{array}{ll}
\left\|t^2S_1(t)[\chi_a\wo](\cdot\sqrt{t})\right\|_{\LL{p}}
\kern-0.6em & = 
t^{2-\frac{1}{p}}\|S_1(t)[\chi_a\wo]\|_{\LL{p}}
\\\rule[1.3em]{0pt}{0pt} & \hspace{-0.7em}\underset{\eqref{propS1r}}{\leqslant} 
2C \woo{r^2(r+|z|)\chi_a}\times t^{-\frac{1}{2}} \xrightarrow[t\to+\infty]{} 0.
\end{array}
\]
We see as well that $t^{2-\frac{1}{p}}\|S(t)[(1-\chi_a)\wo]\|_{\LL{p}} \leqslant C\,\|r^2\wo(1-\chi_a)\|_{\LL{1}}$ using \eqref{46} with $\alpha=0$ and $\beta=2$, then writing
\[
t^2S_1(t)\wo(r\sqrt{t},z\sqrt{t}) = t^2S(t)[\wo-\chi_a\wo](r\sqrt{t},z\sqrt{t}) + t^2S_1(t)[\chi_a\wo](r\sqrt{t},z\sqrt{t}) + (I_0^a-I_0)G_1(r,z)
\]
we finally have
\begin{align*}
\underset{t\to+\infty}{\limsup}\ \|t^2S_1(t)\wo(\cdot\sqrt{t})\|_{\LL{p}}
& \leqslant \underset{t\to+\infty}{\limsup}\ \|t^2S(t)[(1-\chi_a)\wo](\cdot\sqrt{t})\|_{\LL{p}}
\\ &\qquad
+\ \underset{t\to+\infty}{\limsup}\ \|t^2S_1(t)[\chi_a\wo](\cdot\sqrt{t})\|_{\LL{p}}
\\ &\qquad
+\ \underset{t\to+\infty}{\limsup}\ \|(I_0^a-I_0)G_1\|_{\LL{p}}
\\ \rule[1.3em]{0pt}{0pt}&\hskip-0.4em \underset{\eqref{46}}{\leqslant} 
C\,\|r^2\wo(1-\chi_a)\|_{\LL{1}} + 0 + \|(I_0^a-I_0)G_1\|_{\LL{p}}
\end{align*}
which tends to zero as $a\to+\infty$.
\end{proof}

\subsection{Approximation of the semigroup at second order}

Let us suppose briefly that $I_0 = \int_\Omega r^2\wo(r,z)drdz$ vanishes. In this case $S_1(t) = S(t)$ for every $t>0$, and estimate \eqref{propS1r} taken with $\alpha=0$ and $q=p=1$ implies that the semigroup decreases in $\LL{1}$ at least like $t\mapsto t^{-\beta/2}$ when $(r,z)\mapsto (1+r^{\beta}+zr^{\beta-1})\wo(r,z)$ is in $\LL{1}$, if  $\beta\leqslant 3$. The case $\beta=3$ corresponds to the second term in the asymptotic approximation of $S$:
\begin{equation}\label{equivalentS2L1}
S(t)\wo \ \underset{t\to+\infty}{\sim} \ \frac{J_0rz}{32\sqrt{\pi}}e^{\mbox{$-\frac{r^2+z^2}{4t}$}}\, t^{-7/2}
\end{equation}
if $I_0=0$ but $J_0 = \int_\Omega r^2z\,\wo(r,z)\,drdz\neq0$. This is the object of lemma \ref{lemmaconvergenceS2} below. 

Let us introduce on $\Omega$ the Gaussian-like function
\begin{equation}\label{defG2}
G_2(r,z) = -\partial_zG_1(r,z) = \frac{rz}{32\sqrt{\pi}}e^{\mbox{$-\frac{r^2+z^2}{4}$}},
\end{equation}
and the notation
\begin{multline}\label{defS2}
S_2(t)\wo : (r,z) \mapsto S(t)\wo(r,z) - \left(\int_\Omega \rho^2\wo(\rho,\zeta)\,d\rho d\zeta\right)G_1\left(\frac{r}{\sqrt{t}},\frac{z}{\sqrt{t}}\right)t^{-2}\\
\rule[1.7em]{0pt}{0pt} - \left(\int_\Omega \rho^2\zeta\,\wo(\rho,\zeta)\,d\rho d\zeta\right)\,G_2\left(\frac{r}{\sqrt{t}},\frac{z}{\sqrt{t}}\right)t^{-5/2}
\end{multline}
defined when $\woo{(1+r^2+r^2|z|)}<\infty$ and $t>0$. Like $S_1$, the family of operators $(S_2(t))_{t>0}$ is not a semigroup nor is it strongly continuous at $t=0$ on any $\LL{p}$.

\refstepcounter{theorem}
\paragraph{Remark \thetheorem.}\label{rkproptyG1G2}
$G_1$ and $G_2$ defined by \eqref{defG1} and \eqref{defG2} are such that the first term \eqref{equivalentS1L1} and the second term \eqref{equivalentS2L1} in the asymptotic expansion of the semigroup are both self-similar solutions to the linearised vorticity equation
\begin{equation}\label{vorlin}
\partial_t\w = \Delta\w - \frac{\w}{r^2},
\end{equation}
meaning that the equalities
\begin{equation}\label{rkproptyG1}
S(t-s)\left[\frac{1}{s^2}\ G_1\!\left(\frac{\cdot}{\sqrt{s}}\right)\right](r,z) = \frac{1}{t^2}\ G_1\!\left(\frac{r}{\sqrt{t}},\frac{z}{\sqrt{t}}\right)
\end{equation}
and
\begin{equation}\label{rkproptyG2}
S(t-s)\left[\frac{1}{s^{5/2}}\ G_2\!\left(\frac{\cdot}{\sqrt{s}}\right)\right](r,z) = \frac{1}{t^{5/2}}\ G_2\!\left(\frac{r}{\sqrt{t}},\frac{z}{\sqrt{t}}\right)
\end{equation}
hold for all positive times $t\geqslant s>0$. Other notable properties are the values of their $r^2$ and $r^2z$-moments:
\begin{equation}\label{rkG1moment}
\int_\Omega r^2G_1(r,z)\,drdz = 1 \qquad\text{ and }\qquad \int_\Omega r^2z\,G_1(r,z)\,drdz = 0
\end{equation}
while
\begin{equation}\label{rkG2moment}
\int_\Omega r^2G_2(r,z)\,drdz = 0 \qquad\text{ and }\qquad \int_\Omega r^2z\,G_2(r,z)\,drdz = 1.
\end{equation}

\paragraph{}
Let us now see the estimates that hold for $(S_2(t))_{t>0}$, similar to the ones of proposition \ref{propS1} upon $S_1$. Once again, the constants in proposition \ref{propS2} can be shown independent of $p$, $q$, $\alpha$, $\beta$ and $\wo$.

\begin{proposition}\label{propS2}
Choose $1\leqslant p\leqslant q \leqslant \infty$ such that $p<2$ and $-1\leqslant\alpha\leqslant \beta \leqslant 2$.
\begin{itemize}
\item[(i)] If $3-\frac{2}{p}<\beta\leqslant2$
then there exists a positive constant $C$ such that for all time $t>0$
\begin{equation}\label{propS2r}
\|r^\alpha S_2(t)\wo\|_{\LL{q}} 
\leqslant \frac{C}{\frac{1}{p}+\frac{\beta+1}{2}-2}\, \|r^\beta(r^2+z^2)\,\wo\|_{\LL{p}} 
\times t^{-1+\frac{\alpha-\beta}{2}+\frac{1}{q}-\frac{1}{p}}
\end{equation}
provided that $(r,z)\mapsto r^2(1+|z|)\wo(r,z)\in\LL{1}$ and $(r,z)\mapsto r^\beta(r^2+z^2)\wo(r,z)\in\LL{p}$.
\item[(ii)] If $2-\frac{2}{p}<\beta\leqslant1$
then there exists a positive constant $C$ such that for all time $t>0$
\begin{equation}\label{propS2divr}
\|r^\alpha S_2(t)\dii w\|_{\LL{q}} 
\leqslant \frac{C}{\frac{1}{p}+\frac{\beta+2}{2}-2}\, \|r^\beta(r^2+z^2)\,w\|_{\LL{p}} 
\times t^{-\frac{3}{2}+\frac{\alpha-\beta}{2}+\frac{1}{q}-\frac{1}{p}}
\end{equation}
provided that $(r,z)\mapsto r^2(1+|z|)\dii w(r,z)\in\LL{1}$ and $(r,z)\mapsto r^\beta(r^2+z^2)w(r,z)\in(\LL{p})^2$.
\end{itemize}
\end{proposition}

\begin{proof}
The proof is strictly identical to the one of proposition \ref{propS1} once taken into account the following modification and its consequences. The asymptotic expansion with integral remainder \eqref{demopropS1Taylor} has to be considered here at order two
\begin{equation}\label{demopropS2Taylor}
K\left(\frac{r\rho}{\sqrt{t}}\right) e^{-X(t)}
= \frac{\sqrt{\pi}}{4}e^{-\frac{1}{4}(r^2 + z^2)} + \frac{z\zeta\sqrt{\pi}}{8\sqrt{t}}e^{-\frac{1}{4}(r^2 + z^2)} + \int_t^{+\infty} \frac{1}{2s^{3/2}}\left(\frac{1}{\sqrt{t}}-\frac{1}{\sqrt{s}}\right) A_0\,e^{-X(s)}ds,
\end{equation}
$A_0 
$ being now the quantity
\[
\begin{array}{rl}
A_0 = \kern-0.6em
& r^2\rho^2 K''\!\left(\frac{r\rho}{\sqrt{s}}\right) + \dfrac{r\rho}{2}\left(\rho\left(r-\frac{\rho}{\sqrt{s}}\right) + \zeta\left(z-\frac{\zeta}{\sqrt{s}}\right)\right) K'\!\left(\frac{r\rho}{\sqrt{s}}\right)
\\ \el{1.7em} & \qquad 
+ \dfrac{1}{4} \left(\rho\left(r-\frac{\rho}{\sqrt{s}}\right) + \zeta\left(z-\frac{\zeta}{\sqrt{s}}\right)\right)^2 K\!\left(\frac{r\rho}{\sqrt{s}}\right) - \dfrac{1}{2}\left(\rho^2+\zeta^2\right) K\!\left(\frac{r\rho}{\sqrt{s}}\right)
\end{array}
\]
instead of \eqref{demopropS1A0}. Equalities \eqref{demopropS1S1} and \eqref{demopropS1S1div} change accordingly, meaning that their left hand side members involve here $S_2(t)$ instead of $S_1(t)$ and that the integrands have to be multiplied by the factor $(\frac{1}{\sqrt{t}}-\frac{1}{\sqrt{s}})$, which is less than $\frac{1}{\sqrt{t}}$. Last consequence to take into account, $A_r$ and $A_z$ such as defined in \eqref{demopropS1ArAz} and $r\rho^2 A_0$ are now sums of terms of the form
\[
\frac{r^{1+i}\rho^{1+j}}{s^{\frac{j-1}{2}}} \rho^k\left(r-\frac{\rho}{\sqrt{s}}\right)^{k'} \zeta^{2-k}\left(z-\frac{\zeta}{\sqrt{s}}\right)^{k''} K^{(i)}\left(\frac{r\rho}{\sqrt{s}}\right)
\]
with $j\in\{0,1\}$, $k\in\{0,1,2\}$, $i,k',k''\in\{0,1,2,3\}$, and must be bounded by $C s^{\frac{2-\beta}{2}}\rho^\beta(\rho^2+\zeta^2)e^{\frac{X(s)}{2}}$ instead of $C s^{\frac{2-\beta}{2}}\rho^\beta(\rho+|\zeta|)e^{\frac{X(s)}{2}}$. Using this bound inside of \eqref{demopropS2Taylor}, together with a Fubini integral permutation and Young's inequality for convolution, leads to the desired conclusion.
\end{proof}

We can now come back to \eqref{equivalentS2L1}, which is true in every $\LL{p}$, $1\leqslant p\leqslant\infty$.

\begin{lemma}\label{lemmaconvergenceS2}
Take $\wo\in\LL{1}$ such that $(r,z)\mapsto (r^3+r^2|z|)\wo(r,z)\in\LL{1}$. Remember that $(S(t))_{t\geqslant0}$ is the semigoup defined by \eqref{defSt}, $I_0= \int_\Omega r^2\wo(r,z)\,drdz$ and $J_0= \int_\Omega r^2z\,\wo(r,z)\,drdz$. Then
\begin{equation}\label{lemmaconvergenceS2convergence}
S(t)\wo = \frac{I_0r}{16\sqrt{\pi}}e^{\mbox{$-\frac{r^2+z^2}{4}$}}\ t^{-5/2} + \frac{J_0rz}{32\sqrt{\pi}}e^{\mbox{$-\frac{r^2+z^2}{4}$}}\ t^{-7/2} + \underset{t\to+\infty}{o}\left(t^{\frac{1}{p}-\frac{5}{2}}\right)
\end{equation}
in $\LL{p}$ for every $p\in[1,+\infty]$.
\end{lemma}

\begin{proof}
The goal is to prove that $\|t^{5/2}S_2(t)\wo(\cdot\sqrt{t})\|_{\LL{p}}$ tends to zero when $t$ goes to infinity for every $p$ in $[1,+\infty]$. Let us recall the definitions of $G_1$, $S_1$, $G_2$ and $S_2$, respectively in \eqref{defG1}, \eqref{defS1}, \eqref{defG2} and \eqref{defS2}.

Let us consider for any $a>1$ the indicator function $\chi_a$ of the truncated subset $[\frac{1}{a},a]\times[-a,a]$ of $\Omega$, and the quantity $J_0^a = \int_{1/a}^a\int_{-a}^a \rho^2\zeta\,\wo(\rho,\zeta)\,d\rho d\zeta = \int_\Omega\rho^2\zeta\,\chi_a\wo d\rho d\zeta$ which tends towards $J_0$ as $a$ goes to infinity. Let us use \eqref{propS1r} with $\alpha=0$ and $\beta=2$ to show that
\[
\|t^{\frac{5}{2}}S_1(t)[(1-\chi_a)\wo](\cdot\sqrt{t})\|_{\LL{p}} = t^{\frac{5}{2}-\frac{1}{p}}\|S_1(t)[(1-\chi_a)\wo]\|_{\LL{p}} \leqslant 2C \woo{r^2(r+|z|)(1-\chi_a)}
\]
for any positive time $t$, and use \eqref{propS2r} with $\alpha=0$ and $\beta=2$ to see that
\[
\begin{array}{ll}
\|t^{\frac{5}{2}}S_2(t)[\chi_a\wo](\cdot\sqrt{t})\|_{\LL{p}}
\kern-0.6em &
= t^{\frac{5}{2}-\frac{1}{p}}\|S_2(t)[\chi_a\wo]\|_{\LL{p}}
\\\rule[1.3em]{0pt}{0pt} & \hspace{-0.7em}\underset{\eqref{propS2r}}{\leqslant} 
2C \woo{r^2(r^2+z^2)\chi_a}\times t^{-\frac{1}{2}} \xrightarrow[t\to+\infty]{} 0.
\end{array}
\]
Let us write eventually for every $a>1$
\[
t^{\frac{5}{2}}S_2(t)\wo(r\sqrt{t},z\sqrt{t})
= t^{\frac{5}{2}}S_1(t)[(1-\chi_a)\wo](r\sqrt{t},z\sqrt{t})
+ t^{\frac{5}{2}}S_2(t)[\chi_a\wo](r\sqrt{t},z\sqrt{t})
+ (J_0^a-J_0)G_2(r,z)
\]
so
\begin{align*}
\underset{t\to+\infty}{\limsup}\ \|t^{\frac{5}{2}}S_2(t)\wo(\cdot\sqrt{t})\|_{\LL{p}}
& \leqslant \underset{t\to+\infty}{\limsup}\ \|t^{\frac{5}{2}}S_1(t)[(1-\chi_a)\wo](\cdot\sqrt{t})\|_{\LL{p}}
\\ &\qquad
+\ \underset{t\to+\infty}{\limsup}\ \|t^{\frac{5}{2}}S_2(t)[\chi_a\wo](\cdot\sqrt{t})\|_{\LL{p}}
\\ &\qquad
+\ \underset{t\to+\infty}{\limsup}\ \|(J_0^a-J_0)G_2\|_{\LL{p}}
\\ \rule[1.3em]{0pt}{0pt}&\hskip-0.6em \underset{\eqref{propS1r}}{\leqslant} 
 C\,\|r^2(r+|z|)\wo(1-\chi_a)\|_{\LL{1}}
+\,0 + \|(J_0^a-J_0)G_2\|_{\LL{p}}
\end{align*}
which tends to zero as $a\to+\infty$, concluding the proof.
\end{proof}

\section{Time-asymptotic expansion of the vorticity}\label{SectionAsymptoticExpansion}

In \cite{Gallay1}, Gallay and \v Sver\'ak show important results about the axisymmetric vorticity when the swirl is zero: for any nontrivial initial data $\wo\in\LL{1}$, the $L^1$-norm of the vorticity $t\mapsto\wwt{}$ decreases to zero and for every $p$ in $]1,+\infty]$
\begin{equation}\label{73}
\underset{t\to0}{\lim}\ t^{1-\frac{1}{p}} \|\w(t)\|_{\LL{p}} = \underset{t\to+\infty}{\lim}\ t^{1-\frac{1}{p}} \|\w(t)\|_{\LL{p}} = 0.
\end{equation}
As we said in the introduction, they also show that when the initial data $\wo$ is nonnegative and has a finite impulse $I_0 = \int_\Omega r^2\wo(r,z)drdz <\infty$, the vorticity has the same first-order asymptotic expansion as the linear semigroup $S$.

The goal of this section is to refine these results by reinforcing gradually the hypotheses on the localisation of the initial data, proving in particular that the nonnegativity of $\wo$ is never required. Let us start by introducing the Biot-Savart law for axisymmetric fluids, which gives crucial estimates between $u$ and $\w$.

\subsection{The Axisymmetric Biot-Savart law}
\label{sectionBiotSavart}

As stated in the introductory section, working with the vorticity or directly with the velocity is equivalent given that to any solution $\omega$ of the vorticity equation \eqref{voraxi} corresponds only one solution $u$ to the Navier-Stokes equations \eqref{NSdimensionless} whose curl is $\omega$. In the context of axisymmetry without swirl, the Biot-Savart law can be written as
\begin{equation}\label{BSaxi}
\begin{array}{l}
u_r = -\frac{1}{r}\partial_z\psi
\\ \el{1.3em}
u_\theta = 0
\\ \el{1.3em}
u_z\ = \frac{1}{r}\partial_r\psi
\end{array}
\quad \text{ where }\quad
\psi(r,z) = \frac{1}{4\pi}\int_\Omega\int_0^{2\pi} \frac{r\rho\cos\theta\,\w}{\sqrt{(z-\zeta)^2 + r^2 + \rho^2 - 2r\rho\cos\theta}}d\theta d\rho d\zeta,
\end{equation}
see for example \cite{FukMoff00}. The reader will then find the following proposition in \cite{Gallay1}.

\begin{proposition}\label{propBS}
Consider $u$ defined from $\w$ by the axisymmetric Biot-Savart law \eqref{BSaxi} and choose $1\leqslant p<q \leqslant \infty$.
\begin{itemize}
\item[(i)] If $1<p<q<\infty$ such that $\frac{1}{p} = \frac{1}{q} + \frac{1}{2}$ and $\w\in\LL{p}$ then there exists a positive constant $C$ such that for all time $t\geqslant0$
\begin{equation}\label{25}
\|u(t)\|_{\LL{q}} \leqslant C \wwtp{}{p}.
\end{equation}
More generally, if $1<p<q<\infty$ and $0\leqslant\alpha\leqslant\beta\leqslant2$ such that $0\leqslant\beta-\alpha<1$ and $\frac{1}{p} = \frac{1}{q} + \frac{1+\alpha-\beta}{2}$ and if $(r,z) \mapsto r^\beta\w(r,z)\in\LL{p}$ then there exists a positive constant $C$ such that for all time $t\geqslant0$
\begin{equation}\label{28}
\|r^\alpha u(t)\|_{\LL{q}} \leqslant C \wwtp{r^\beta}{p}.
\end{equation}
\item[(ii)] If $1\leqslant p<2<q \leqslant \infty$ and $\w\in\LL{p}\cap\LL{q}$ then there exists a positive constant $C$ such that for all time $t\geqslant0$
\begin{equation}\label{26}
\uut \leqslant C \wwtp{}{p}^\sigma\wwtp{}{q}^{1-\sigma} \qquad\text{where } \sigma = \frac{p}{2}\frac{q-2}{q-p} \in\ ]0,1[\,.
\end{equation}
\end{itemize}
\end{proposition}

\paragraph{Remark.}
The combination of \eqref{73} and \eqref{26} ensures the existence of two positive constants $C$  and $C_{\wo}$ such that
\begin{equation}\label{uLinftyw}
\uut \leqslant C \wwt{}^{\frac{1}{2}} \left(C_{\wo}\, t^{-1}\right)^{\frac{1}{2}} \qquad \forall t>0,
\end{equation}
and the fact that $t\mapsto\wwt{}$ is nonincreasing gives the existence of a positive constant $C_{\wo}$ such that
\begin{equation}\label{uLinfty}
\uut \leqslant C_{\wo}\,t^{-\frac{1}{2}} \qquad \forall t>0.
\end{equation}
This bound will be useful in what comes next, and already tells something about how $u$ vanishes at infinity.

\subsection{Decay rates for the vorticity}
\label{sectiondecayrates}

Let us show first that a finite initial impulse $\woo{r^2}<\infty$ implies that the second radial moment $\wwt{r^2}$ is bounded and that $\wwtp{}{p}$ decreases like $t\mapsto t^{\frac{1}{p}-2}$; this is lemma \ref{lemmamajorationsr2} below. 
This result requires the intermediate study of the vorticity's decay rates when only its first radial moment $\woo{r}$ is initially supposed finite.

\begin{lemma}\label{lemmamajorationsr}
Take $\w$ the solution of \eqref{vor} with initial condition $\wo\in\LL{1}$.\\
If $(r,z)\mapsto r\wo(r,z)$ is in $\LL{1}$ then there exists a positive constant $C_{\wo}$ such that for every $q\in[1,+\infty]$ and all time $t>0$
\begin{align}
\label{lemmamajorationsrw}
\wwtp{}{q} & \leqslant C_{\wo}\,(1+t)^{-\frac{1}{2}}t^{\frac{1}{q}-1}
\\ \label{lemmamajorationsrrw}
\wwtp{r}{q} & \leqslant C_{\wo}\,t^{\frac{1}{q}-1}. \el{1.3em}
\end{align}
\end{lemma}

\begin{proof}
Suppose that $\woo{r}<\infty$. We show first the local existence of $\wwt{r}$, then the fact that it does not blow up in finite time and eventually that it is bounded on $[0,+\infty[$. We shall come afterwards to inequalities \eqref{lemmamajorationsrw} and \eqref{lemmamajorationsrrw}.

First of all, a fixed point argument shows that there exists a time $T>0$ such that $\w$ is continuous from $[0,T]$ into the space
\[
L^1_1(\Omega) = \{w:\Omega\to\RR \mid \|w\|_{\LL{1}} + \|rw\|_{\LL{1}} <\infty\}.
\]
This standard argument will be used several times in the next lemmas, we thus provide some details in appendix \ref{annexeprop4.1} for the reader's convenience. 

We then expose why $\wwt{r}$ does not blow up in finite time, using a Grönwall argument. Let us write for $t\geqslant0$ and $t_0>0$
\begin{align}
&\wwtt{r}{t_0+t}
\notag\\ \rule{0pt}{0pt}& \hskip5em\displaystyle \underset{\eqref{duhameluni}}{\leqslant} 
\|rS(t)\w(t_0)\|_{\LL{1}} + \int_{t_0}^{t_0+t}\|rS(t_0+t-s)\dii(u(s)\w(s))\|_{\LL{1}}\,ds
\notag\\ \rule[2.2em]{0pt}{0pt}& \hskip4em\displaystyle \underset{\eqref{46}\eqref{47}}{\leqslant} 
C\wwtt{r}{t_0} + C\int_{t_0}^{t_0+t}(t_0+t-s)^{-\frac{1}{2}}\uus\wws{r}\,ds
\label{local4}\\ \el{2.2em}& \hskip5em\displaystyle \underset{\eqref{uLinfty}}{\leqslant} 
C\wwtt{r}{t_0} + C\int_0^t(t-s)^{-\frac{1}{2}}(t_0+s)^{-\frac{1}{2}}\wwtt{r}{t_0+s}\,ds \notag 
\end{align}
and conclude using some Grönwall-type lemma that $\wwt{r}$ is finite and well defined for every time $t$ in $[0,+\infty[$. For example, lemma 7.1.1 from \cite{DanHenry} gives exactly the desired conclusion if we observe that $(t_0+s)^{-\frac{1}{2}}$ is bounded by the constant $t_0^{-\frac{1}{2}}$.

Let us now come back to line \eqref{local4} and give a bound to $\uus$ by using inequality \eqref{uLinftyw} instead of \eqref{uLinfty}, then $\uus\leqslant C\wws{}^{\frac{1}{2}}s^{-\frac{1}{2}}$ and inequality \eqref{local4} becomes
\[
\wwtt{r}{t_0+t} \leqslant C\wwtt{r}{t_0} + C_2\int_{t_0}^{t_0+t}(t_0+t-s)^{-\frac{1}{2}}\wws{}^{\frac{1}{2}}s^{-\frac{1}{2}}\wws{r}\,ds
\]
for any $t\geqslant0$, $C_2$ being some positive constant depending on $\woo{}$. Let us remember here that $t\mapsto\wwt{}$ decreases to zero and consider $t_0$ large enough so that $\wwtt{}{t_0} \leqslant (2C_2\pi)^{-2}$; the presence of $\pi$ being due to the particular value of the Euler Beta function $B(\frac{1}{2},\frac{1}{2}) = \int_0^1(1-s)^{-\frac{1}{2}}s^{-\frac{1}{2}}ds = \pi$ appearing right below. Indeed, for every $t>0$ we can now write 
\begin{align*}
&\wwtt{r}{t_0+t}
\\& \hskip3em \leqslant 
C\wwtt{r}{t_0} + C_2\,\wwtt{}{t_0}^{\frac{1}{2}} \int_0^{t_0+t}(t_0+t-s)^{-\frac{1}{2}}s^{-\frac{1}{2}}\,ds \left(\underset{[t_0,t_0+t]}{\sup}\|r\w\|_{\LL{1}}\right)
\\& \hskip3em 
\leqslant C\wwtt{r}{t_0} + \frac{1}{2}\ \underset{[t_0,t_0+t]}{\sup}\|r\w\|_{\LL{1}}
\end{align*}
and so $\sup_{[t_0,t_0+t]}\|r\w\|_{\LL{1}} \leqslant C + \frac{1}{2}\sup_{[t_0,t_0+t]}\|r\w\|_{\LL{1}}$. This proves that there exists a positive constant $C_{\wo}$ such that $\wwt{r} \leqslant C_{\wo}$ for any $t$ in $[0,+\infty[$.

We can now prove inequalities \eqref{lemmamajorationsrw} and \eqref{lemmamajorationsrrw}, which are consequences of this uniform bound on $\wwt{r}$ as shown below. To establish \eqref{lemmamajorationsrw}, let us recall first that $\wwtp{}{q} \leqslant C_{\wo}\,t^{\frac{1}{q}-1}$ as a consequence of \eqref{73}. Referring then to \eqref{46} and \eqref{47} with $\alpha=0$ and $\beta=\frac{1}{2}$, let us write for $t>0$
\[
\begin{array}{ll}
\wwt{}
&\hskip-0.6em\displaystyle \underset{\eqref{duhamel}}{\leqslant} 
\|S(t)\wo\|_{\LL{1}} + \int_0^t\|S(t-s)\dii(u(s)\w(s))\|_{\LL{1}}\,ds
\\ \el{2.2em} &\hskip-1.6em\displaystyle \underset{\eqref{46}\eqref{47}}{\leqslant} 
C\woo{r^{\frac{1}{2}}}t^{-\frac{1}{4}} + C\int_0^t(t-s)^{-\frac{3}{4}}\uus\wws{r^{\frac{1}{2}}}\,ds
\\ \el{2.2em} &\hskip-0.6em\displaystyle \underset{\eqref{uLinfty}}{\leqslant} 
Ct^{-\frac{1}{4}} + C\int_0^t(t-s)^{-\frac{3}{4}}s^{-\frac{1}{2}}\,ds
\quad =\ C(1 + B(1/4,1/2))\,t^{-\frac{1}{4}}.
\end{array}
\]
Use straightaway the decay rate obtained this way on $\wwt{}$ in the same calculation (referring this time to \eqref{46} with $\alpha=0$, $\beta=1$ and to \eqref{47} with $\alpha=0$, $\beta=\frac{3}{4}$),
\begin{align*}
\wwt{}
&\underset{\eqref{duhamel}}{\leqslant} 
\|S(t)\wo\|_{\LL{1}} + \int_0^t\|S(t-s)\dii(u(s)\w(s))\|_{\LL{1}}\,ds
\\ \el{2.2em} &\hskip-1em \underset{\eqref{46}\eqref{47}}{\leqslant} 
C\woo{r}t^{-\frac{1}{2}} + C\int_0^t(t-s)^{-\frac{7}{8}}\uus\wws{r^{\frac{3}{4}}}\,ds
\\ \el{2.2em} &\underset{\eqref{uLinftyw}}{\leqslant} 
Ct^{-\frac{1}{2}} + C\int_0^t(t-s)^{-\frac{7}{8}}\wws{}^{\frac{1}{2}}s^{-\frac{1}{2}}\wws{r^{\frac{3}{4}}}\,ds
\\ \el{2.2em} &\hskip0.4em 
\leqslant Ct^{-\frac{1}{2}} + C\int_0^t(t-s)^{-\frac{7}{8}}s^{-\frac{1}{8}}s^{-\frac{1}{2}}\,ds
\quad =\ C(1 + B(1/8,3/8))\,t^{-\frac{1}{2}},
\end{align*}
to get the case $q=1$ of \eqref{lemmamajorationsrw}. Finally, another iterative argument gives all of the remaining cases in three steps. Let us proceed to the following calculations three times, successively for $q\in[1,2[$, $p=1$, then $q\in[\frac{4}{3},4[$, $p=\frac{4}{3}$ and at last $q\in[3,+\infty]$, $p=3$:
\begin{align*}
\wwtp{}{q}
& \underset{\eqref{duhameluni}}{\leqslant} 
\|S(t/2)\w(t/2)\|_{\LL{q}} + \int_{t/2}^t\|S(t-s)\dii(u(s)\w(s))\|_{\LL{q}}\,ds
\\ \el{2.2em} &\hskip-1em \underset{\eqref{46}\eqref{47}}{\leqslant} 
C\wwtt{r}{t/2}t^{\frac{1}{q}-\frac{3}{2}} + C\int_{t/2}^t(t-s)^{\frac{1}{q}-\frac{1}{p}-\frac{1}{2}}\uus\wwsp{}{p}\,ds
\\ \el{2.2em} & \underset{\eqref{uLinfty}}{\leqslant} 
Ct^{\frac{1}{q}-\frac{3}{2}} + C\int_{t/2}^t(t-s)^{\frac{1}{q}-\frac{1}{p}-\frac{1}{2}}s^{-\frac{1}{2}}\wwsp{}{p}\,ds
\\ \el{2.2em} &\hskip0.4em 
\leqslant Ct^{\frac{1}{q}-\frac{3}{2}} + C\int_{t/2}^t(t-s)^{\frac{1}{q}-\frac{1}{p}-\frac{1}{2}}s^{-\frac{1}{2}}s^{\frac{1}{p}-\frac{3}{2}}\,ds
\\ \el{2.2em} &\hskip0.4em 
\leqslant C\left(1 + \int_{1/2}^{1}(1-s)^{\frac{1}{q}-\frac{1}{p}-\frac{1}{2}}s^{\frac{1}{p}-2}ds\right)t^{\frac{1}{q}-\frac{3}{2}}
\end{align*}
and
\begin{align*}
\wwtp{r}{q}
& \underset{\eqref{duhameluni}}{\leqslant} 
\|rS(t/2)\w(t/2)\|_{\LL{q}} + \int_{t/2}^t\|rS(t-s)\dii(u(s)\w(s))\|_{\LL{q}}\,ds
\\ \el{2.2em} &\hskip-1em \underset{\eqref{46}\eqref{47}}{\leqslant} 
C\wwtt{r}{t/2}t^{\frac{1}{q}-1} + C\int_{t/2}^t(t-s)^{\frac{1}{q}-\frac{1}{p}-\frac{1}{2}}\uus\wwsp{r}{p}\,ds
\\ \el{2.2em} & \underset{\eqref{uLinfty}}{\leqslant} 
Ct^{\frac{1}{q}-1} + C\int_{t/2}^t(t-s)^{\frac{1}{q}-\frac{1}{p}-\frac{1}{2}}s^{-\frac{1}{2}}\wwsp{r}{p}\,ds
\\ \el{2.2em} &\hskip0.4em 
\leqslant Ct^{\frac{1}{q}-1} + C\int_{t/2}^t(t-s)^{\frac{1}{q}-\frac{1}{p}-\frac{1}{2}}s^{-\frac{1}{2}}s^{\frac{1}{p}-1}\,ds
\\ \el{2.2em} &\hskip0.4em 
\leqslant C\left(1 + \int_{1/2}^{1}(1-s)^{\frac{1}{q}-\frac{1}{p}-\frac{1}{2}}s^{\frac{1}{p}-\frac{3}{2}}ds\right)t^{\frac{1}{q}-1}
\end{align*}
for $t>0$. Note that the constants in the computations only depend on the initial data $\wo$.
\end{proof}

\paragraph{Remark.}
Under the hypothesis of $\woo{r}<\infty$, \eqref{lemmamajorationsrw} is an refinement of \eqref{73}. It allows us to replace \eqref{uLinfty} by the following estimate
\begin{equation}\label{uLinftyr}
\uut \leqslant C(1+t)^{-1/2}t^{-1/2} \qquad \forall t>0,
\end{equation}
obtained directly thanks to \eqref{lemmamajorationsrw} combined with \eqref{26}. This bound will be used exclusively to prove the next lemma, before being superseded by a better one. 

\begin{lemma}\label{lemmamajorationsr2}
Take $\w$ the solution of \eqref{vor} with initial condition $\wo\in\LL{1}$.\\
If $(r,z)\mapsto r^2\wo(r,z)$ is in $\LL{1}$ then there exists a positive constant $C_{\wo}$ such that for every $q\in[1,+\infty]$ and all time $t>0$
\begin{align}
\label{lemmamajorationsr2w}
\wwtp{}{q} & \leqslant C_{\wo}\,(1+t)^{-1}t^{\frac{1}{q}-1}
\\ \label{lemmamajorationsr2rw}
\wwtp{r}{q} & \leqslant C_{\wo}\,(1+t)^{-\frac{1}{2}}t^{\frac{1}{q}-1} \el{1.3em}
\\ \label{lemmamajorationsr2r2w}
\wwtp{r^2}{q} & \leqslant C_{\wo}\,t^{\frac{1}{q}-1}. \el{1.3em}
\end{align}
\end{lemma}

\begin{proof}
Suppose that $\woo{r^2}<\infty$. As for lemma \ref{lemmamajorationsr}, we shall show first the local existence of the quantity $\wwt{r^2}$, then the fact that it does not blow up in finite time and eventually that it is bounded on $[0,+\infty[$. We shall afterwards concentrate on inequalities \eqref{lemmamajorationsr2w} to \eqref{lemmamajorationsr2r2w}. Note that, in particular, $\woo{r}<\infty$ so the conclusions of lemma \ref{lemmamajorationsr} hold.

First of all, a fixed point argument shows that there exists a time $T>0$ such that $\w$ is continuous from $[0,T]$ into
\[
\{w:\Omega\to\RR \mid \|w\|_{\LL{1}} + \|r^2 w\|_{\LL{1}} <\infty\},
\]
read for example the development in appendix \ref{annexeprop4.1} but adapted with the norm
\[
\|\w\|_{X_T} = \underset{0<t\leqslant T}{\sup}\left(t^{1/4}\wwtp{}{4/3} + t^{1/4}\wwtp{r^2}{4/3}\right)
\]
on $X_T$ and using \eqref{propSdivr2} together with an interpolation inequality when adapting equations \eqref{54} and \eqref{55}.

Let us now show that $\wwt{r^2}$ does not blow up in finite time. Let us write for $t\geqslant0$ that
\begin{equation}\label{local6}
\begin{split}
\wwt{r}
& \underset{\eqref{duhameluni}}{\leqslant} 
\|rS(t/2)\w(t/2)\|_{\LL{1}} + \int_{t/2}^t\|rS(t-s)\dii(u(s)\w(s))\|_{\LL{1}}\,ds
\\ &\hskip-1em \underset{\eqref{46}\eqref{47}}{\leqslant} 
C\wwtt{r^2}{t/2}t^{-\frac{1}{2}} + C\int_{t/2}^t(t-s)^{-\frac{1}{2}}\uus\wws{r}\,ds
\\ &\hskip-1em \underset{\eqref{uLinftyr}\eqref{lemmamajorationsrrw}}{\leqslant} 
C\wwtt{r^2}{t/2}t^{-\frac{1}{2}} + C\int_{t/2}^t(t-s)^{-\frac{1}{2}}(1+s)^{-\frac{1}{2}}s^{-\frac{1}{2}}\,ds
\\ &\hskip0.4em  
\leqslant C\left(\wwtt{r^2}{t/2} + \int_{1/2}^1(1-s)^{-\frac{1}{2}}s^{-1}\,ds\right) t^{-\frac{1}{2}},
\end{split}
\end{equation}
which leads for $t\geqslant0$ and $t_0>0$ to
\begin{align*}
&\wwtt{r^2}{t_0+t}
\\& \hskip5em \underset{\eqref{duhameluni}}{\leqslant} 
\|r^2S(t)\w(t_0)\|_{\LL{1}} + \int_{t_0}^{t_0+t}\|r^2S(t_0+t-s)\dii(u(s)\w(s))\|_{\LL{1}}\,ds
\\& \hskip4em \underset{\eqref{46}\eqref{propSdivr2}}{\leqslant} 
C\wwtt{r^2}{t_0} + C\int_{t_0}^{t_0+t}\uus\wws{r},ds
\\& \hskip5em \qquad
+ C\int_{t_0}^{t_0+t}(t_0+t-s)^{-\frac{1}{2}}\|ru(s)\|_{\LL{4}}\wwsp{r}{4/3}\,ds
\\\el{1.7em}& \hskip3.2em \underset{\eqref{local6}\eqref{28}\eqref{lemmamajorationsrrw}}{\leqslant} 
C\wwtt{r^2}{t_0} + C\int_{t_0}^{t_0+t}(1+s)^{-\frac{1}{2}}s^{-\frac{1}{2}}(\wwtt{r^2}{s/2}+1)s^{-\frac{1}{2}}\,ds
\\& \hskip5em \qquad
+ C\int_{t_0}^{t_0+t}(t_0+t-s)^{-\frac{1}{2}}s^{-\frac{1}{2}}\,ds
\end{align*}
where in the last line \eqref{28} is used with $\alpha=\beta=1$ and $p=4$, $q=\frac{4}{3}$. Translating the integration variable $s$ by $t_0$, using that $(t_0+s)^{-1/2}\leqslant\min(1,t_0)^{-1/2}(1+s)^{-1/2}$ and in the first integrand that $\wwtt{r^2}{s/2} \leqslant \sup_{[0,t_0]}\|r^2\w\|_{\LL{1}} + \sup_{[t_0,s]}\|r^2\w\|_{\LL{1}}$, one gets for every $t\geqslant t_0>0$ that
\[
\underset{[t_0,t]}{\sup}\|r^2\w\|_{\LL{1}} \leqslant C_{\wo,t_0} + C_{t_0}\int_0^t (1+s)^{-\frac{3}{2}} \underset{[t_0,s]}{\sup}\|r^2\w\|_{\LL{1}}\,ds.
\]
Grönwall's inequality applied to the map $t\mapsto\sup_{s\in[t_0,t]}\wws{r^2}$ therefore ensures that $t\mapsto\wwt{r^2}$ is well-defined and bounded on $[0,+\infty[$. This sets estimate \eqref{lemmamajorationsr2r2w} when $q=1$.

We can now thanks to this uniform bound prove inequalities \eqref{lemmamajorationsr2w} to \eqref{lemmamajorationsr2r2w}. To establish \eqref{lemmamajorationsr2w} and \eqref{lemmamajorationsr2rw}, one has to recall first that $\wwtp{}{q} \leqslant C_{\wo}\,t^{\frac{1}{q}-1}$ due to \eqref{73} and $\wwtp{r}{q} \leqslant C_{\wo}\,t^{\frac{1}{q}-1}$ due to \eqref{lemmamajorationsrrw}. Then, let us write for $t>0$
\begin{align*}
\wwt{}
& \underset{\eqref{duhameluni}}{\leqslant} 
\|S(t/2)\w(t/2)\|_{\LL{1}} + \int_{t/2}^t\|S(t-s)\dii(u(s)\w(s))\|_{\LL{1}}\,ds
\\ \el{2.2em} &\hskip-1em \underset{\eqref{46}\eqref{47}}{\leqslant} 
C\wwtt{r^2}{t/2}t^{-1} + C\int_{t/2}^t(t-s)^{-\frac{1}{2}}\uus\wws{}\,ds
\\ \el{2.2em} &\hskip-1em \underset{\eqref{uLinftyr}\eqref{lemmamajorationsrw}}{\leqslant} 
Ct^{-1} + C\int_{t/2}^t(t-s)^{-\frac{1}{2}}(1+s)^{-\frac{1}{2}}s^{-\frac{1}{2}}(1+s)^{-\frac{1}{2}}\,ds
\\ \el{2.2em} &\hskip0.4em  
\leqslant C\left(1 + \int_{1/2}^1(1-s)^{-\frac{1}{2}}s^{-\frac{3}{2}}\,ds\right)t^{-1}
\end{align*}
to get the case $q=1$ of \eqref{lemmamajorationsr2w}, and let us remember calculation \eqref{local6} which gives the case $q=1$ of \eqref{lemmamajorationsr2rw} once we know that $\|r^2\w\|_{\LL{1}}$ is bounded on $[0,+\infty[$. Finally, a three-steps iterative argument leads to all of the remaining cases. Let us proceed to the following computations three times, successively for $q\in[1,2[$, $p=1$ then $q\in[\frac{4}{3},4[$, $p=\frac{4}{3}$ and at last $q\in[3,+\infty]$, $p=3$:
\begin{align*}
\wwtp{}{q}
& \underset{\eqref{duhameluni}}{\leqslant} 
\|S(t/2)\w(t/2)\|_{\LL{q}} + \int_{t/2}^t\|S(t-s)\dii(u(s)\w(s))\|_{\LL{q}}\,ds
\\ \el{1.7em} &\hskip-1em \underset{\eqref{46}\eqref{47}}{\leqslant} 
C\wwtt{r^2}{t/2}t^{\frac{1}{q}-2} + C\int_{t/2}^t(t-s)^{\frac{1}{q}-\frac{1}{p}-\frac{1}{2}}\uus\wwsp{}{p}\,ds
\\ \el{2.2em} & \underset{\eqref{uLinfty}}{\leqslant} 
Ct^{\frac{1}{q}-2} + C\int_{t/2}^t(t-s)^{\frac{1}{q}-\frac{1}{p}-\frac{1}{2}}s^{-\frac{1}{2}}\wwsp{}{p}\,ds
\\ \el{2.2em} &\hskip0.4em  
\leqslant Ct^{\frac{1}{q}-2} + C\int_{t/2}^t(t-s)^{\frac{1}{q}-\frac{1}{p}-\frac{1}{2}}s^{-\frac{1}{2}}s^{\frac{1}{p}-2}\,ds
\\ \el{2.2em} &\hskip0.4em  
\leqslant C\left(1 + \int_{1/2}^1(1-s)^{\frac{1}{q}-\frac{1}{p}-\frac{1}{2}}s^{\frac{1}{p}-\frac{5}{2}}\,ds\right)t^{\frac{1}{q}-2}
\end{align*}
and
\begin{align*}
\wwtp{r}{q}
& \underset{\eqref{duhameluni}}{\leqslant} 
\|rS(t/2)\w(t/2)\|_{\LL{q}} + \int_{t/2}^t\|rS(t-s)\dii(u(s)\w(s))\|_{\LL{q}}\,ds
\\ \el{1.7em} &\hskip-1em \underset{\eqref{46}\eqref{47}}{\leqslant} 
C\wwtt{r^2}{t/2}t^{\frac{1}{q}-\frac{3}{2}} + C\int_{t/2}^t(t-s)^{\frac{1}{q}-\frac{1}{p}-\frac{1}{2}}\uus\wwsp{r}{p}\,ds
\\ \el{2.2em} &\underset{\eqref{uLinfty}}{\leqslant} 
Ct^{\frac{1}{q}-\frac{3}{2}} + C\int_{t/2}^t(t-s)^{\frac{1}{q}-\frac{1}{p}-\frac{1}{2}}s^{-\frac{1}{2}}\wwsp{r}{p}\,ds
\\ \el{2.2em} &\hskip0.4em 
\leqslant Ct^{\frac{1}{q}-\frac{3}{2}} + C\int_{t/2}^t(t-s)^{\frac{1}{q}-\frac{1}{p}-\frac{1}{2}}s^{-\frac{1}{2}}s^{\frac{1}{p}-\frac{3}{2}}\,ds
\end{align*}
and
\begin{align*}
\wwtp{r^2}{q}
& \underset{\eqref{duhameluni}}{\leqslant} 
\|r^2S(t/2)\w(t/2)\|_{\LL{q}} + \int_{t/2}^t\|r^2S(t-s)\dii(u(s)\w(s))\|_{\LL{q}}\,ds
\\ \el{1.7em} &\hskip-1em \underset{\eqref{46}\eqref{propSdivr2}}{\leqslant} 
C\wwtt{r^2}{t/2}t^{\frac{1}{q}-1} + C\int_{t/2}^t(t-s)^{\frac{1}{q}-\frac{1}{p}}\uus\wwsp{r}{p}\,ds
\\ &\qquad
+ C\int_{t/2}^t(t-s)^{\frac{1}{q}-\frac{1}{p}-\frac{1}{2}}\uus\wwsp{r^2}{p}\,ds
\\ \el{1.7em} & \underset{\eqref{uLinfty}}{\leqslant} 
Ct^{\frac{1}{q}-1} + C\int_{t/2}^t(t-s)^{\frac{1}{q}-\frac{1}{p}}s^{-\frac{1}{2}}\wwsp{r}{p}\,ds
\\ &\qquad
+ C\int_{t/2}^t(t-s)^{\frac{1}{q}-\frac{1}{p}-\frac{1}{2}}s^{-\frac{1}{2}}\wwsp{r^2}{p}\,ds
\\ \el{1.7em} &\hskip0.4em  
\leqslant Ct^{\frac{1}{q}-1} + C\int_{t/2}^t(t-s)^{\frac{1}{q}-\frac{1}{p}}s^{-\frac{1}{2}}s^{\frac{1}{p}-\frac{3}{2}}\,ds + C\int_{t/2}^t(t-s)^{\frac{1}{q}-\frac{1}{p}-\frac{1}{2}}s^{-\frac{1}{2}}s^{\frac{1}{p}-1}\,ds,
\end{align*}
where the constants only depend on the initial data $\wo$.
\end{proof}

\paragraph{Remark.}
In particular, we deduce from the Biot-Savart law that if $\woorz{r^2} < \infty$ then there is some positive constant $C$ such that for all time $t>0$
\begin{equation}\label{uLinftyr2}
\|u(t)\|_{\LL{\infty}} \leqslant C(1+t)^{-1}t^{-1/2}.
\end{equation}
This estimate is given by inequalities \eqref{26} and \eqref{lemmamajorationsr2w} put together, and it is the best one can get concerning the supremum of $u$ as long as the impulse $I_0$ is not zero. 

\paragraph{}
Lemmas \ref{lemmamajorationsr} and \ref{lemmamajorationsr2} are the most direct way to reach the case where $(r,z)\mapsto(1+r^2)\wo(r,z)$ is in $\LL{1}$, but one could consider all of the intermediate cases:

\begin{proposition}\label{propmajorationsralpha}
Take $\w$ the solution of \eqref{vor} with initial condition $\wo\in\LL{1}$.\\
Take $\alpha$ in $[0,2]$. If $(r,z)\mapsto r^\alpha\wo(r,z)$ is in $\LL{1}$ then there exists a positive constant $C_{\wo,\alpha}$ such that for every $q\in[1,+\infty]$ and all time $t>0$
\begin{align}
\label{propmajorationsralphaw}
\wwtp{}{q} & \leqslant C_{\wo,\alpha}\,(1+t)^{-\frac{\alpha}{2}}t^{\frac{1}{q}-1}
\\ \label{propmajorationsralphawralpha}
\wwtp{r^\alpha}{q} & \leqslant C_{\wo,\alpha}\,t^{\frac{1}{q}-1}. \el{1.3em}
\end{align}
\end{proposition}

\begin{proof}
The arguments presented in the proofs of lemmas \ref{lemmamajorationsr} and \ref{lemmamajorationsr2} are applicable in the same order, following the proof of lemma \ref{lemmamajorationsr} when $\alpha\in[0,1]$ and the proof of lemma \ref{lemmamajorationsr2} when $\alpha\in\ ]1,2]$. 
\end{proof}

\subsection{First-Order asymptotics}
\label{sectionfirstorder}

In this subsection, we show first \eqref{theoremintror2} then we study how close $\w(t)$ and its first asymptotic term are the one to the other.

Due to the symmetry of the problem, when both $\wo$ and $(r,z)\mapsto r^2\wo(r,z)$ are in $\LL{1}$ the vorticity has a constant impulse: 
\begin{equation}\label{ItisI0}
I(t) = \int_\Omega r^2\w(t,r,z)\,drdz = \int_\Omega r^2\wo(r,z)\,drdz = I_0 \qquad\forall t\geqslant0.
\end{equation}
Indeed, remembering that $\Delta = \partial_r^2 + \frac{1}{r}\partial_r + \partial_z^2$ in \eqref{vor} implies that
\[
\dot{I}(t) = \int_\Omega r^2\partial_t\w(t) = \int_\Omega r^2\left( \partial_r^2 + \partial_z^2 + \frac{1}{r}\partial_r - \frac{1}{r^2}\right)\w(t) -\int_\Omega r^2\dii(u(t)\w(t))
\]
where after several integrations by parts, the first integral vanishes and the second is equal to $\int_\Omega 2r\,u_r(t)\w(t)$.
Writing therefore $\dot{I}(t) = \int_\Omega 2r\,u_r(t)\w(t)$,
then substituting $u_r$ with its expression by the Biot-Savart law \eqref{BSaxi} shows that $\dot{I}(t)$ is zero, because the integrand under the double integral obtained that way is odd with respect to permuting the variables $(r,z)$ and $(\rho,\zeta)$ in $\Omega\times\Omega$. 

\paragraph{Remark.}
The existence of $\dot{I}(t) = \frac{d}{dt}I(t)$ is ensured by the fact that $\w\in\CC^1(]0,+\infty[,L^1_2(\Omega))$, where $L^1_2(\Omega) = L^1(\Omega,(1+r^2)drdz)$. The reader can refer to the end of appendix \ref{annexeprop4.1} for more details.

\begin{proposition}\label{propwr2}
Take $\w$ the solution of \eqref{vor} with initial data $\wo\in\LL{1}$.\\
If $(r,z)\mapsto r^2\wo(r,z)$ is in $\LL{1}$ then
\begin{equation}\label{propwr2convergence}
\w(t,r,z) \ = \ \frac{I_0r}{16\sqrt{\pi}}e^{\mbox{$-\frac{r^2+z^2}{4t}$}}\ t^{-5/2} + \underset{t\to+\infty}{o}\left(t^{\frac{1}{p}-2}\right)
\end{equation}
in $\LL{p}$ for every $p\in[1,+\infty]$.
\end{proposition}

\begin{proof}
Let $\woo{r^2}$ be finite, $\wo\in\LL{1}$. The goal is to show that, in the self-similar coordinates $(r\sqrt{t},z\sqrt{t})$, the quantity $t^2\w(t,\cdot\sqrt{t},\cdot\sqrt{t}) - I_0G_1$ tends to zero in any $\LL{p}$ when $t$ tends to infinity (remembering that $G_1$ is given by \eqref{defG1}). 

Let us fix $t_0>0$, then for any $p$ in $[1,+\infty]$ and any $t> 2t_0$ we can write 
\begin{align*}
\|t^2\w(t,\cdot\sqrt{t},\cdot\sqrt{t}) - I_0G_1\|_{\LL{p}}
\underset{\eqref{duhameluni}}{\leqslant} &
\left\|t^2S(t-t_0)[\w(t_0)](\cdot\sqrt{t}) - \left(\frac{t}{t-t_0}\right)^2\!\!I_0G_1\!\left(\!\cdot\,\frac{\sqrt{t}}{\sqrt{t-t_0}}\right)\right\|_{\LL{p}}
\\ \rule[2.2em]{0pt}{0pt} &
+ \left\|I_0\left(1-\frac{t_0}{t}\right)^{-2}G_1\left(\cdot\left(1-\frac{t_0}{t}\right)^{-\frac{1}{2}}\right) - I_0G_1\right\|_{\LL{p}}
\\ \rule[2.2em]{0pt}{0pt} &
+ \int_{t_0}^{t/2}t^2\left\|S(t-s)\dii(u(s)\,\w(s))\left(\cdot\sqrt{t}\right)\right\|_{\LL{p}}ds
\\ \rule[2.2em]{0pt}{0pt} &
+ \int_{t/2}^tt^2\left\|S(t-s)\dii(u(s)\,\w(s))\left(\cdot\sqrt{t}\right)\right\|_{\LL{p}}ds
\end{align*}
where three of those four terms tend to zero as $t$ tends to infinity and one will have to be dominated by a decreasing function of $t_0$. 

The convergence of the second term is trivial. The first one can be written with a change of variables as $(\frac{t}{t-t_0})^{2-\frac{1}{p}}\|(t-t_0)^2S(t-t_0)[\w(t_0)](\cdot\sqrt{t-t_0}) - I_0G_1\|_{\LL{p}}$, which tends to zero according to lemma \ref{lemmaconvergenceS} and the fact that $I(t_0)=I_0$ (recall \eqref{ItisI0}).  The fourth one satisfies
\begin{align*}
&\int_{t/2}^tt^2\left\|S(t-s)\dii(u(s)\,\w(s))\left(\cdot\sqrt{t}\right)\right\|_{\LL{p}}ds
\\& \hskip9.8em \underset{\eqref{47}}{\leqslant} 
C\int_{t/2}^t t^{2-\frac{1}{p}}(t-s)^{-\frac{1}{2}}\uus{}\wwsp{}{p}\,ds
\\& \hskip8.8em \underset{\eqref{uLinftyr2}\eqref{lemmamajorationsr2w}}{\leqslant} 
C\int_{t/2}^t t^{2-\frac{1}{p}}(t-s)^{-\frac{1}{2}}s^{-\frac{3}{2}}s^{\frac{1}{p}-2}\,ds
\\& \hskip10.4em  
\leqslant C\left(\int_{1/2}^{1}(1-s)^{-\frac{1}{2}}s^{\frac{1}{p}-\frac{7}{2}}ds\right) t^{-1}
\end{align*}
so converges to zero as well, and the third term can be bounded as
\begin{align*}
&\int_{t_0}^{t/2}t^2\left\|S(t-s)\dii(u(s)\,\w(s))\left(\cdot\sqrt{t}\right)\right\|_{\LL{p}}ds
\\& \hskip9.8em \underset{\eqref{47}}{\leqslant} 
C\int_{t_0}^{t/2} t^{2-\frac{1}{p}}(t-s)^{\frac{1}{p}-2}\uus{}\wws{r}\,ds
\\& \hskip8.8em \underset{\eqref{uLinftyr2}\eqref{lemmamajorationsr2rw}}{\leqslant} 
C\int_{t_0}^{t/2}2^{2-\frac{1}{p}} (1+s)^{-1}s^{-\frac{1}{2}}(1+s)^{-\frac{1}{2}}\,ds
\\& \hskip10.4em 
\leqslant C\left(\int_0^{+\infty}(1+s)^{-\frac{3}{2}}ds\right) t_0^{-\frac{1}{2}}
\end{align*}
where the constants do not depend on $t_0$. Finally, 
\[
\underset{t\to+\infty}{\limsup}\ \|t^2\w(t,\cdot\sqrt{t},\cdot\sqrt{t}) - I_0G_1\|_{\LL{p}} \leqslant 0 + 0 + C\,t_0^{-\frac{1}{2}} + 0
\]
for any $t_0>0$, which ends the proof.
\end{proof}

Let us now specify the $o(t^{\frac{1}{p}-2})$ in proposition \ref{propwr2} by reinforcing the hypotheses on $\wo$. For this, let us remark that another consequence of 
\begin{equation}\label{I't=0}
\dot{I}(t) = -\int_\Omega r^2\dii(u(t,r,z)\w(t,r,z))\,drdz= 0
\end{equation}
is that for every $t\geqslant s\geqslant0$
\begin{equation}\label{S1divuwisSdivuw}
S_1(t-s)\dii(u(s)\w(s)) = S(t-s)\dii(u(s)\w(s)),
\end{equation}
where $S_1$ is defined by \eqref{defS1}. 
This means that the difference $\w(t) - I_0G_1(\frac{\cdot}{\sqrt{t}})t^{-2}$ between the vorticity and its first asymptotic term satisfies 
\begin{equation}\label{duhameldifference}
\w(t) - I_0G_1\left(\frac{\cdot}{\sqrt{t}}\right)t^{-2} \underset{\eqref{duhamel}\eqref{S1divuwisSdivuw}}{=} S_1(t)\wo - \int_0^tS_1(t-s)\dii(u(s)\w(s))\,ds.
\end{equation}

\begin{proposition}\label{propwr2+alpha}
Take $\w$ the solution of \eqref{vor} with initial condition $\wo\in\LL{1}$.\\
Take $\alpha$ in $]0,1]$. If $(r,z)\mapsto r^\alpha(r^2+r|z|+|z|)\wo(r,z)$ is in $\LL{1}$ then there exists a positive constant $C_{\wo,\alpha}$ such that for every $q\in[1,+\infty]$ and all $t>0$
\begin{equation}
\left\|\,\w(t,r,z) - \frac{I_0r}{16\sqrt{\pi}}e^{\mbox{$-\frac{r^2+z^2}{4t}$}}\ t^{-\frac{5}{2}}\right\|_{\LL{q}} 
\leqslant C_{\wo,\alpha}\,t^{\frac{1}{q}-2-\frac{\alpha}{2}}.
\end{equation}
\end{proposition}

Before showing proposition \ref{propwr2+alpha}, we need a lemma certifying that some of the finite moments asked to $\wo$ stay finite through the evolution of $\w$.

\begin{lemma}\label{lemmamajorationr1alpha+ralphaz}
Take $\w$ the solution of \eqref{vor} with initial condition $\wo\in\LL{1}$.\\
Take $\alpha$ in $[0,1]$. If $(r,z)\mapsto r^\alpha(r+|z|)\wo(r,z)$ is in $\LL{1}$ then $t\mapsto\wwtrz{r^\alpha(r+|z|)\,}$ is bounded on $[0,+\infty[$.
\end{lemma}

\begin{proof}
Let $\woo{r^\alpha(r+|z|)\,}$ be finite, where $0\leqslant\alpha\leqslant1$. The strategy of this proof will be similar to those of lemmas \ref{lemmamajorationsr} and \ref{lemmamajorationsr2}. $\woo{r^{1+\alpha}}$ being supposed finite, we already have access to the conclusions of proposition \ref{propmajorationsralpha}, especially the fact that by combining \eqref{26} and \eqref{propmajorationsralphaw} we know that $\uut \leqslant C(1+t)^{-\frac{1+\alpha}{2}}t^{-1/2}$ for all time $t>0$.

First of all, a fixed point argument shows that there exists a time $T>0$ such that $\w$ is continuous from $[0,T]$ into
\[
\{w:\Omega\to\RR \mid \|w\|_{\LL{1}} + \|rw\|_{\LL{1}} + \|r^\alpha zw\|_{\LL{1}} <\infty\},
\]
read for example the development in appendix \ref{annexeprop4.1} adapted with the norm
\[
\|\w\|_{X_T} = \underset{0<t\leqslant T}{\sup}\left(t^{1/4} \wwtp{}{4/3} + t^{1/4}\wwtp{r}{4/3} + t^{1/4}\wwtp{r^\alpha z}{4/3}\right)
\]
on $X_T$ and using \eqref{propSdivralphaz} together with interpolations inequalities when adapting equations \eqref{54} and \eqref{55}.

Let us then show that $\wwt{r^\alpha z}$ is well-defined and bounded on $[0,+\infty[$, by referring to lemma \ref{lemmaGronwallbis-gamma} in appendix and using the estimate $\uus \leqslant C(1+s)^{-\frac{1+\alpha}{2}}s^{-1/2}$ as follows. For every $t\geqslant0$ and $t_0>0$, 
\begin{align*}
\wwtt{r^\alpha z}{t_0+t}
&\underset{\eqref{duhameluni}}{\leqslant} 
\|r^\alpha zS(t)\w(t_0)\|_{\LL{1}} + \int_{t_0}^{t_0+t}\|r^\alpha zS(t_0+t-s)\dii(u(s)\w(s))\|_{\LL{1}}\,ds
\\ &\hskip-1.2em \underset{\eqref{propSralphaz}\eqref{propSdivralphaz}}{\leqslant} 
C\wwtt{r^\alpha(r+|z|)\,}{t_0}
\\ &\qquad
+ C\int_{t_0}^{t_0+t}(t_0+t-s)^{\frac{\alpha-1}{2}}\uus\wws{r}\,ds
\\ &\qquad
+ C\int_{t_0}^{t_0+t}(t_0+t-s)^{-\frac{1}{2}}\uus\wws{r^\alpha z}\,ds
\\ &\underset{\eqref{lemmamajorationsrrw}}{\leqslant} 
C + C\int_{t_0}^{t_0+t}(t_0+t-s)^{\frac{\alpha-1}{2}}(1+s)^{-\frac{1+\alpha}{2}}s^{-\frac{1}{2}}\,ds
\\ &\qquad 
+ C\int_{t_0}^{t_0+t}(t_0+t-s)^{-\frac{1}{2}}(1+s)^{-\frac{1+\alpha}{2}}s^{-\frac{1}{2}}\wws{r^\alpha z}\,ds
\end{align*}
where $\int_0^{t_0+t}(t_0+t-s)^{\frac{\alpha-1}{2}}(1+s)^{-\frac{1+\alpha}{2}}s^{-\frac{1}{2}}\,ds \leqslant 10$ and where, after the change of variable $s\mapsto t_0+s$ in the last integral, we can write $(1+t_0+s)^{-\frac{1+\alpha}{2}} \leqslant (1+s)^{-\frac{1}{2}}$ and $(t_0+s)^{-\frac{1}{2}} \leqslant \min(1,t_0)^{-\frac{1}{2}} (1+s)^{-\frac{1}{2}}$. 
Hence 
\[
\wwtt{r^\alpha z}{t_0+t} \leqslant C + C\int_0^{t}(t-s)^{-\frac{1}{2}}(1+s)^{-1}\wwtt{r^\alpha z}{t_0+s}\,ds,
\]
and we can apply lemma \ref{lemmaGronwallbis-gamma}. Let us note that the constants only depend on $\wo$ and the choice of $t_0$, but not on $\alpha$.
\end{proof}

\begin{proof}[Proof of proposition \ref{propwr2+alpha}]
Let us recall the definitions of $G_1$ and $S_1$, given by  \eqref{defG1} and \eqref{defS1}. Using lemma \ref{lemmamajorationr1alpha+ralphaz} and knowing \eqref{S1divuwisSdivuw}, let us carry out the following calculation, starting from \eqref{duhameldifference}:
\begin{align*}
&\left\|\w(t) - I_0G_1\left(\frac{\cdot}{\sqrt{t}}\right)\ t^{-2}\right\|_{\LL{q}} 
\\& \hskip11.8em \underset{\eqref{duhameldifference}}{\leqslant} 
\|S_1(t)\wo\|_{\LL{q}} + \int_0^{t/2}\|S_1(t-s)\dii(u(s)\w(s))\|_{\LL{q}}\,ds
\\& \hskip12em \qquad
+ \int_{t/2}^t\|S(t-s)\dii(u(s)\w(s))\|_{\LL{q}}\,ds
\\\el{2.2em}& \hskip9.8em \underset{\eqref{propS1r}\eqref{propS1divr}\eqref{47}}{\leqslant} 
\frac{2C}{\alpha} \woo{r^\alpha(r^2+r|z|)\,} t^{\frac{1}{q}-\frac{4+\alpha}{2}}
\\& \hskip12em \qquad
+ \int_0^{t/2}\frac{2C}{\alpha}(t-s)^{\frac{1}{q}-\frac{4+\alpha}{2}}\uus{}\wws{r^\alpha(r+|z|)\,}\,ds
\\& \hskip12em \qquad
+ \int_{t/2}^tC(t-s)^{-\frac{1}{2}}\uus{}\wwsp{}{q}\,ds 
\\& \hskip10.8em \underset{\eqref{uLinftyr2}\eqref{lemmamajorationsr2w}}{\leqslant} 
C\,t^{\frac{1}{q}-\frac{4+\alpha}{2}} + C \left(\frac{t}{2}\right)^{\frac{1}{q}-\frac{4+\alpha}{2}} \int_0^{+\infty}(1+s)^{-1}s^{-\frac{1}{2}}\,ds
\\& \hskip12em \qquad
+ C\int_{t/2}^t(t-s)^{-\frac{1}{2}}s^{-1}s^{\frac{1}{q}-\frac{3+\alpha}{2}}\,ds
\\\el{1.3em}& \hskip12.4em \leqslant 
C\,t^{\frac{1}{q}-2-\frac{\alpha}{2}}
\end{align*}
for any $t>0$. Let us note that the last constant depends on the initial data $\wo$ and on the quantity $\frac{1}{\alpha}$.
\end{proof}

Considering the limit case $\alpha=1$ in proposition \ref{propwr2+alpha}, in particular $\woorz{r^2z}$ and $\woorz{r(r+|z|)\,}$ are finite so the moment $J(t) = \int_\Omega r^2z\,\w(t,r,z)\,drdz$ of the vorticity is bounded.

\begin{corollary}
Take $\w$ the solution of \eqref{vor} with initial condition $\wo\in\LL{1}$.\\
If $(r,z)\mapsto (r^2|z|+r^2+r|z|)\wo(r,z)$ is in $\LL{1}$, then $J$ is bounded on $[0,+\infty[$. To be more precise, $J$ is differentiable and tends to the finite limit
\begin{equation}\label{defJinfty}
J_\infty = J_0+\int_0^{+\infty}\int_\Omega(2rz\,u_r(t,r,z) + r^2u_z(t,r,z))\,\w(t,r,z)\,drdzdt
\end{equation}
when time tends to infinity.
\end{corollary}

\begin{proof}
It has to be said first that $\w$ is differentiable in time into $L^1(\Omega,(1+r^2+r^2|z|)drdz)$, for reasons given in remark \ref{rkC1}. Then $J$ is differentiable and, referring to the vorticity equation \eqref{vor} and using that $\dii(u\w) = \partial_r(u_r\w) + \partial_z(u_z\w)$, a direct calculation gives its evolution for all time $t\geqslant0$:
\begin{equation}\label{J't}
\dot{J}(t) = \int_\Omega r^2z\,\partial_t\w(t) = -\int_\Omega r^2z\,\dii(u(t)\w(t)) = \int_\Omega (2rz\,u_r + r^2u_z)\,\w.
\end{equation}
Hence $|\dot{J}(t)| \leqslant 3\|u(t)\|_{\LL{\infty}}\,\wwt{r(r+|z|)\,}$, and from lemma \ref{lemmamajorationr1alpha+ralphaz} and estimate \eqref{uLinftyr2} follows directly the existence of a positive constant $C_{\wo}$ such that
\begin{equation}\label{majorationJ'}
|\dot{J}(t)| \leqslant C_{\wo}(1+t)^{-1}t^{-1/2},
\end{equation}
so $J$ is bounded on $[0,+\infty[$ and the expression of $J_\infty$ \eqref{defJinfty} holds.
\end{proof}

\subsection{Second-Order asymptotics}

At that point, we want to express the next term in the asymptotic expansion of $\w$. The object of the next propositions is to write for all $(t,r,z)$ in $]0,+\infty[*\times\Omega$ 
\begin{equation}\label{defwtilde}
\w(t,r,z) = \frac{I_0r}{16\sqrt{\pi}}e^{\mbox{$-\frac{r^2+z^2}{4t}$}}\ t^{-5/2} + \frac{J(t)rz}{32\sqrt{\pi}}e^{\mbox{$-\frac{r^2+z^2}{4t}$}}\ t^{-7/2} + \widetilde{\w}(t,r,z)
\end{equation}
where the remainder $\widetilde{\w}(t,r,z)$ is a $o(t^{\frac{1}{p}-\frac{5}{2}})$ in every $\LL{p}$ when $t$ goes to infinity (this is proposition \ref{propwr3}) and even a $O(\ln(t)t^{\frac{1}{p}-3})$ if enough spatial localisation is asked on the initial data (this is proposition \ref{propwr4}). Before moving towards the demonstration of these results, there are a few things to say about $\widetilde{\w}$.

Given that $\widetilde{\w}$ defined by \eqref{defwtilde} is a difference between $\CC^\infty$ maps, it is itself in $\CC^\infty(]0,+\infty[\times\Omega)$. Moreover, given that $\w$ satisfies the vorticity equation \eqref{vor} and that $(t,r,z)\mapsto G_1(\frac{r}{\sqrt{t}},\frac{z}{\sqrt{t}}) t^{-2}$ and $(t,r,z)\mapsto G_2(\frac{r}{\sqrt{t}},\frac{z}{\sqrt{t}}) t^{-\frac{5}{2}}$ are solutions to the linearised equation \eqref{vorlin}, 
then for every $t>0$
\begin{equation}\label{voromegatilda}
\partial_t\widetilde{\w} = \Delta\widetilde{\w} - \frac{\widetilde{\w}}{r^2} - \dii(u\w) - \frac{\dot{J}(t)}{t^{5/2}}G_2\left(\frac{\cdot}{\sqrt{t}}\right)
\end{equation}
which is equivalent to the integral formulation
\begin{equation}\label{duhamelomegatilde}
\widetilde{\w}(t) = S_2(t)\wo - \int_0^tS(t-s)\left[ \dii(u(s)\w(s)) + \frac{\dot{J}(s)}{s^{5/2}}G_2\left(\frac{\cdot}{\sqrt{s}}\right) \right]ds
\end{equation}
where $S_2$ is the family of operators defined as \eqref{defS2} and $G_1$ and $G_2$ are the Gaussian-like functions \eqref{defG1} and \eqref{defG2}.

Let us pay special attention to the various possible ways to write the integrand in equation \eqref{duhamelomegatilde}, because this will be helpful in further calculation. First, let us remember that $(t,r,z)\mapsto G_2(\frac{r}{\sqrt{t}},\frac{z}{\sqrt{t}})t^{-\frac{5}{2}}$ can be taken out of the semigroup, as noted in \eqref{rkproptyG2}, and especially that $\int_{t_0}^{t_1}S(t-s)[\dot{J}(s)G_2(\frac{\cdot}{\sqrt{s}})s^{-\frac{5}{2}}]\,ds = (J(t_1)-J(t_0))G_2(\frac{\cdot}{\sqrt{t}})t^{-\frac{5}{2}}$, whatever $0\leqslant t_0\leqslant t_1\leqslant t$ are. Let us notice that it can also be put inside the $\dii$ by writing
\begin{equation}\label{G2isadiv}
\frac{\dot{J}(s)}{s^{5/2}}G_2\left(\frac{r}{\sqrt{s}},\frac{z}{\sqrt{s}}\right) = \dii\left(\frac{\dot{J}(s)}{32\sqrt{\pi}\,s^2}\begin{pmatrix}\frac{-z}{\sqrt{s}}\\\frac{-r}{\sqrt{s}}\end{pmatrix}e^{\mbox{$-\frac{1}{4s}(r^2+z^2)$}} \right).
\end{equation}
Finally, in view of the fact that $\int_\Omega r^2\dii(u(s)\w(s)) = 0$, that $\int_\Omega r^2z\dii(u(s)\w(s)) = -\dot{J}(s)$ and that $\int_\Omega r^2zG_2 =1$ (see respectively \eqref{I't=0}, \eqref{J't} and \eqref{rkG2moment}), the relation
\begin{equation}\label{S2divuw+J'G2}
S(t-s)\!\left[\dii(u(s)\w(s)) + \frac{\dot{J}(s)}{s^{5/2}}G_2\left(\frac{\cdot}{\sqrt{s}}\right) \right] = S_2(t-s)\!\left[\dii(u(s)\w(s)) + \frac{\dot{J}(s)}{s^{5/2}}G_2\left(\frac{\cdot}{\sqrt{s}}\right) \right]
\end{equation}
holds on $\Omega$ for any $t\geqslant s>0$. Obviously, \eqref{S2divuw+J'G2} remains true by substituting $S_1$ for $S_2$.

This being said, let us start the study of $\widetilde{\w}$ under the hypotheses of proposition \eqref{propwr2+alpha} when $\alpha=1$.

\begin{proposition}\label{propwr3}
Take $\w$ the solution of \eqref{vor} with initial data $\wo\in\LL{1}$.\\
If $(r,z)\mapsto (r^3 + r^2|z| + r|z|)\wo(r,z)$ is in $\LL{1}$ then
\begin{equation}\label{propwr3convergence}
\w(t,r,z) = \frac{rI_0}{16\sqrt{\pi}}e^{\mbox{$-\frac{r^2+z^2}{4t}$}}\ t^{-5/2} + \frac{rzJ(t)}{32\sqrt{\pi}}e^{\mbox{$-\frac{r^2+z^2}{4t}$}}\ t^{-7/2} + \underset{t\to+\infty}{o}\left(t^{\frac{1}{p}-\frac{5}{2}}\right)
\end{equation}
in $\LL{p}$ for every $p\in[1,+\infty]$.
\end{proposition}

\begin{proof}
In view of \eqref{defwtilde}, the goal is to show that $t^{\frac{5}{2}}\,\widetilde{\w}(t,\cdot\sqrt{t},\cdot\sqrt{t})$ tends to zero in $\LL{p}$ as $t$ goes to infinity. Let us recall for this the definitions of $G_1$, $S_1$, $G_2$ and $S_2$ from \eqref{defG1}, \eqref{defS1}, \eqref{defG2}, and \eqref{defS2}.

Let us fix any $t_0>0$ and see that, given property \eqref{rkproptyG2} of $G_2$, \eqref{duhamelomegatilde} can as well be written as
\begin{multline*}
\widetilde{\w}(t) = 
S(t-t_0)\w(t_0) - I_0\,G_1\!\left(\frac{\cdot}{\sqrt{t}}\right) t^{-2} - J(t_0)\,G_2\!\left(\frac{\cdot}{\sqrt{t}}\right) t^{-\frac{5}{2}}\\
- \int_{t_0}^tS(t-s)\left[\dii(u(s)\w(s)) + \frac{\dot{J}(s)}{s^{5/2}} G_2\!\left(\frac{\cdot}{\sqrt{s}}\right)\right] ds
\end{multline*}
for every time $t > t_0$. Therefore, for any $p$ in $[1,+\infty]$ and $t> 2t_0$, $\|t^{\frac{5}{2}}\widetilde{\w}(t,\cdot\sqrt{t},\cdot\sqrt{t})\|_{\LL{p}}$ is less than or equal to the sum of the six terms 
\begin{align*}
&
T_1(t,t_0,p) = \|t^{\frac{5}{2}}S_2(t-t_0)[\w(t_0)](\cdot\sqrt{t})\|_{\LL{p}},
\\&
T_2(t,t_0,p) = \left\|I_0\sqrt{t}\left(1-\frac{t_0}{t}\right)^{-2}G_1\left(\cdot\left(1-\frac{t_0}{t}\right)^{-\frac{1}{2}}\right) - I_0\sqrt{t}\,G_1\right\|_{\LL{p}},
\\& 
T_3(t,t_0,p) = \left\|J(t_0)\left(1-\frac{t_0}{t}\right)^{-\frac{5}{2}}G_2\left(\cdot\left(1-\frac{t_0}{t}\right)^{-\frac{1}{2}}\right) - J(t_0)\,G_2\right\|_{\LL{p}},
\\& 
T_4(t,t_0,p) = \int_{t_0}^{t/2}t^{\frac{5}{2}}\left\|S_1(t-s)\dii \left(u(s)\,\w(s) + \frac{\dot{J}(s)}{32\sqrt{\pi}\,s^2}\begin{pmatrix}\frac{-z}{\sqrt{s}}\\\frac{-r}{\sqrt{s}}\end{pmatrix}e^{\mbox{$-\frac{1}{4s}(r^2+z^2)$}}\right)\left(\cdot\sqrt{t}\right)\right\|_{\LL{p}}\!ds,
\\& 
T_5(t,p) = \int_{t/2}^tt^{\frac{5}{2}}\left\|S(t-s)\dii(u(s)\,\w(s))\left(\cdot\sqrt{t}\right)\right\|_{\LL{p}}ds,
\\& 
T_6(t,p) = \int_{t/2}^t|\dot{J}(s)|\times\|G_2\|_{\LL{p}}ds,
\end{align*}
among which five tend to zero when $t$ goes to infinity and one will have to be dominated by a decreasing function of $t_0$. Let us remember what has been discussed right above between equation \eqref{duhamelomegatilde} and relation \eqref{S2divuw+J'G2} to understand why the integrands of $T_4$ and $T_6$ can be written as they are.

The convergence of $T_2$ and $T_3$ to zero is obvious when $p=+\infty$, and is a straightforward application of the Lebesgue dominated convergence theorem when $p<\infty$. The first term $T_1$ equals $(\frac{t}{t-t_0})^{\frac{5}{2}-\frac{1}{p}}\|(t-t_0)^{\frac{5}{2}}S_2(t-t_0)[\w(t_0)](\cdot\sqrt{t-t_0}\,)\|_{\LL{p}}$ with a change of variables, and so converges to zero according to lemma \ref{lemmaconvergenceS2}. The fifth one satisfies
\begin{align*}
T_5
&\kern-0.2em \underset{\eqref{47}}{\leqslant} 
C\int_{t/2}^t t^{\frac{5}{2}-\frac{1}{p}}(t-s)^{-\frac{1}{2}}\uus{}\wwsp{}{p}\,ds
\\& \hskip-1.2em \underset{\eqref{uLinftyr2}\eqref{lemmamajorationsr2w}}{\leqslant} 
C\int_{t/2}^t t^{\frac{5}{2}-\frac{1}{p}}(t-s)^{-\frac{1}{2}}s^{-\frac{3}{2}}s^{\frac{1}{p}-2}\,ds
\\& \kern0.4em \leqslant 
C\left(\int_{1/2}^{1}(1-s)^{-\frac{1}{2}}s^{\frac{1}{p}-\frac{7}{2}}ds\right) t^{-\frac{1}{2}},
\end{align*}
and $T_6 \leqslant Ct^{-\frac{1}{2}}$ given \eqref{majorationJ'}, so both tend to zero. Finally, thanks to lemma \ref{lemmamajorationr1alpha+ralphaz} the fourth term can be bounded as
\begin{align*}
T_4
&\hskip-0.2em \underset{\eqref{propS1divr}}{\leqslant} 
2C\int_{t_0}^{t/2} t^{\frac{5}{2}-\frac{1}{p}}(t-s)^{\frac{1}{p}-\frac{5}{2}}\uus{}\wws{r(r+|z|)}\,ds
\\& \qquad
+ 2C\int_{t_0}^{t/2} t^{\frac{5}{2}-\frac{1}{p}}(t-s)^{\frac{1}{p}-\frac{5}{2}}\frac{|\dot{J}(s)|}{s^2}\left\|r(r+|z|)\begin{pmatrix}\frac{-z}{\sqrt{s}}\\\frac{-r}{\sqrt{s}}\end{pmatrix}e^{\mbox{$-\frac{r^2+z^2}{4s}$}}\right\|_{\LL{1}}ds
\\& \hskip-1.2em \underset{\eqref{uLinftyr2}\eqref{majorationJ'}}{\leqslant} 
4C\int_{t_0}^{t/2}t^{\frac{5}{2}-\frac{1}{p}}\left(\frac{t}{2}\right)^{\frac{1}{p}-\frac{5}{2}} (1+s)^{-1}s^{-\frac{1}{2}}ds
\\& \kern0.4em \leqslant 
C\left(\int_0^{+\infty}(1+s)^{-1}s^{-\frac{1}{4}}ds\right) t_0^{-\frac{1}{4}}
\end{align*}
where the constants do not depend on $t_0$. Thus
\[
\underset{t\to+\infty}{\limsup}\ \|t^{\frac{5}{2}}\widetilde{\w}(t,\cdot\sqrt{t},\cdot\sqrt{t})\|_{\LL{p}} \leqslant 0 + 0 + 0 + C\,t_0^{-\frac{1}{4}} + 0 + 0
\]
for any $t_0>0$, which ends the proof. 
\end{proof}

\refstepcounter{theorem}
\paragraph{Remark \thetheorem.}\label{rkJinfty}
Within the scope of this article, each time the quantity $J(t)$ appears in the expansion of $\w$ it can be replaced by its limit $J_\infty$. This is because, in view of \eqref{majorationJ'}, for all $t>0$
\[
|J(t) - J_\infty| \leqslant \int_t^{+\infty}|\dot{J}(s)|ds  \leqslant C\int_t^{+\infty}s^{-\frac{3}{2}}ds = 2Ct^{-\frac{1}{2}}
\]
so the difference between $J(t)G_2(\frac{\cdot}{\sqrt{t}})t^{-\frac{5}{2}}$ and $J_\infty G_2(\frac{\cdot}{\sqrt{t}})t^{-\frac{5}{2}}$ is a $O(t^{\frac{1}{p}-3})$ in every $\LL{p}$; and it appears that in propositions \ref{propwr3}, \ref{propwr3+alpha} and \ref{propwr4} the precision of the expansions never exceeds a $O(t^{\frac{1}{p}-3})$. This remark, together with the previous proposition, conclude the proof of theorem \ref{theoremintro}.

\paragraph{}
The same way we did in proposition \ref{propwr2+alpha} concerning the difference 
$\w-I_0G_1(\frac{\cdot}{\sqrt{t}})t^{-2}$, the remainder $\widetilde{\w}$ can be shown 
less than a $o(t^{\frac{1}{q}-\frac{5}{2}})$ by assuming more localisation on the initial vorticity $\wo$.

\begin{proposition}\label{propwr3+alpha}
Take $\w$ the solution of \eqref{vor} with initial condition $\wo\in\LL{1}$.\\
Take $\alpha$ in $]0,1[$. If $(r,z)\mapsto r^\alpha(r^3+rz^2+z^2)\wo(r,z)$ is in $\LL{1}$ then there exists a positive constant $C_{\wo,\alpha}$ such that for every $q\in[1,+\infty]$ and all $t>0$ 
\begin{equation}
\|\widetilde{\w}(t)\|_{\LL{q}} \leqslant C_{\wo,\alpha}\,t^{\frac{1}{q}-\frac{5}{2}-\frac{\alpha}{2}}
\end{equation}
where
$\widetilde{\w}$ is given by \eqref{defwtilde}.
\end{proposition}

\begin{lemma}\label{lemmamajorationr2alpha+ralphaz2}
Take $\w$ the solution of \eqref{vor} with initial condition $\wo\in\LL{1}$.\\
Take $\alpha$ in $[0,1]$. If $(r,z)\mapsto r^\alpha(r^2+z^2)\wo(r,z)$ is in $\LL{1}$ then there exists a positive constant $C_{\wo}$ such that for all $t\geqslant0$ 
\begin{equation}\label{lemmamajorationr2alpha+ralphaz2wr2alpha+ralphaz2}
\wwt{r^\alpha(r^2+z^2)\,} \leqslant C_{\wo}(1+t)^{\alpha/2}.
\end{equation}
\end{lemma}

\begin{proof}
Let $\woo{r^\alpha(r^2+z^2)\,}$ be finite, where $0\leqslant\alpha\leqslant1$. Let us show first that the quantity $\wwt{r^\alpha(r^2+z^2)\,}$ exists locally, then that $t\mapsto (1+t)^{-\alpha/2}\wwt{r^\alpha(r^2+z^2)\,}$ is bounded on $[0,+\infty[$ and thus does not blow up in finite time. Given that $\woo{r^2}<\infty$, we have access to the conclusions of lemma \ref{lemmamajorationsr2}.

First of all, a fixed point argument shows that there exists a time $T>0$ such that $\w$ is continuous from $[0,T]$ into
\[
\{w:\Omega\to\RR \mid \|w\|_{\LL{1}} + \|r^{2+\alpha}w\|_{\LL{1}} + \|r^\alpha z^2w\|_{\LL{1}} <\infty\}.
\]
read for example the development in appendix \ref{annexeprop4.1} adapted with the norm
\[
\|\w\|_{X_T} = \underset{0<t\leqslant T}{\sup}\left(t^{1/4} \wwtp{}{4/3} + t^{1/4}\wwtp{r^{2+\alpha}}{4/3} + t^{1/4}\wwtp{r^\alpha z^2}{4/3}\right)
\]
on $X_T$ and using \eqref{propSdivr2alpha} and \eqref{propSdivralphaz2} together with an interpolation inequality when adapting equations \eqref{54} and \eqref{55}.

Let us now show \eqref{lemmamajorationr2alpha+ralphaz2wr2alpha+ralphaz2} with the use of lemma \ref{lemmaGronwallbis-gamma}. Let us write for every $t\geqslant0$ and $t_0>0$ 
\begin{align*} 
& \wwtt{r^\alpha(r^2+z^2)\,}{t_0+t}
\\& \hskip11em \underset{\eqref{duhameluni}}{\leqslant} 
\|r^\alpha(r^2+z^2)S(t)\w(t_0)\|_{\LL{1}} 
\\& \hskip12em \qquad
+ \int_{t_0}^{t_0+t}\|r^\alpha(r^2+z^2)S(t_0+t-s)\dii(u(s)\w(s))\|_{\LL{1}}\,ds
\\& \hskip7.7em \underset{\eqref{propSr2alpha}\eqref{propSralphaz2}\eqref{propSdivr2alpha}\eqref{propSdivralphaz2}}{\leqslant} 
C\wwtt{r^2}{t_0}t^{\frac{\alpha}{2}} + C\wwtt{r^\alpha(r^2+z^2)\,}{t_0}
\\& \hskip12em \qquad
 + C\int_{t_0}^{t_0+t}(t_0+t-s)^{\frac{\alpha}{2}}\uus\wws{r}\,ds
\\& \hskip12em \qquad
+ C\int_{t_0}^{t_0+t}(t_0+t-s)^{-\frac{1}{2}}\uus\wws{r^\alpha(r^2+z^2)\,}\,ds
\\& \hskip9.8em \underset{\eqref{uLinftyr2}\eqref{lemmamajorationsr2rw}}{\leqslant} 
C(1+t^{\frac{\alpha}{2}}) + C\int_{t_0}^{t_0+t}(t_0+t-s)^{\frac{\alpha}{2}}(1+s)^{-1}s^{-\frac{1}{2}}(1+s)^{-\frac{1}{2}}\,ds
\\& \hskip12em \qquad
+ C\int_{t_0}^{t_0+t}(t_0+t-s)^{-\frac{1}{2}}(1+s)^{-1}s^{-\frac{1}{2}}\wws{r^\alpha(r^2+z^2)\,}\,ds
\end{align*}
where $t^{-\frac{\alpha}{2}}\times\int_0^t(t-s)^{\frac{\alpha}{2}}(1+s)^{-\frac{3}{2}}s^{-\frac{1}{2}}\,ds \leqslant 4$ and where, after the change of variable $s\mapsto t_0+s$ in the last integral, $(1+t_0+s)^{-1} \leqslant (1+s)^{-1}$ and $(t_0+s)^{-\frac{1}{2}} \leqslant \min(1,t_0)^{-\frac{1}{2}} (1+s)^{-\frac{1}{2}}$. Thus 
\begin{multline*}
\left\|\frac{r^\alpha(r^2+z^2)}{(t_0+t)^{\frac{\alpha}{2}}}\,\w(t_0+t)\right\|_{\LL{1}} \leqslant C\left(t_0^{-\frac{\alpha}{2}}+1\right) + 4C
\\
+ C \int_0^{t}(t-s)^{-\frac{1}{2}}(1+s)^{-\frac{3}{2}}\left\|\frac{r^\alpha(r^2+z^2)}{(t_0+s)^{\frac{\alpha}{2}}}\,\w(t_0+s)\right\|_{\LL{1}}ds
\end{multline*}
which is a case covered by lemma \ref{lemmaGronwallbis-gamma}. Let us note that the constants only depend on $\wo$ and on the choice of $t_0$. 
\end{proof}

\begin{proof}[Proof of proposition \ref{propwr3+alpha}]
Be sure to have in mind the definitions of $G_1$, $G_2$ and $S_2$ given by \eqref{defG1}, \eqref{defG2} and \eqref{defS2} respectively. Starting from \eqref{duhamelomegatilde} while remembering \eqref{S2divuw+J'G2} and \eqref{G2isadiv}, let us write for all $t>0$
\begin{align*}
\|\widetilde{\w}(t)\|_{\LL{q}} 
& \underset{\eqref{duhamelomegatilde}}{\leqslant}
\|S_2(t)\wo\|_{\LL{q}} 
\\ &\qquad
+ \int_0^{t/2}\left\|S_2(t-s)\dii \left( u(s)\w(s) + \frac{\dot{J}(s)}{32\sqrt{\pi}\,s^2} \begin{pmatrix}\frac{-z}{\sqrt{s}}\\\frac{-r}{\sqrt{s}}\end{pmatrix} e^{\mbox{$-\frac{1}{4s}(r^2+z^2)$}} \right)\right\|_{\LL{q}}\!ds
\\ &\qquad 
+ \int_{t/2}^t\|S(t-s)\dii(u(s)\w(s))\|_{\LL{q}}\,ds
\\ &\qquad 
+ \int_{t/2}^t|\dot{J}(s)|\left\|S(t-s)\left[G_2\left(\frac{\cdot}{\sqrt{s}}\right)s^{-\frac{5}{2}}\right]\right\|_{\LL{q}} ds
\end{align*}
so using respectively \eqref{propS2r}, \eqref{propS2divr}, \eqref{47} and \eqref{majorationJ'} in the four terms
\begin{align*}
\|\widetilde{\w}(t)\|_{\LL{q}} 
&
\leqslant \frac{2C}{\alpha} \woo{r^\alpha(r^3+rz^2)\,}\,t^{\frac{1}{q}-\frac{5+\alpha}{2}}
\\& \qquad 
+ \int_0^{t/2}\frac{2C}{\alpha}(t-s)^{\frac{1}{q}-\frac{5+\alpha}{2}} \uus{}\wws{r^\alpha(r^2+z^2)\,}\,ds 
\\& \qquad 
+ \int_0^{t/2}2C(t-s)^{\frac{1}{q}-\frac{5+\alpha}{2}} \frac{|\dot{J}(s)|}{s^2}\left\| r^\alpha(r^2+z^2)\begin{pmatrix}\frac{z}{\sqrt{s}}\\\frac{r}{\sqrt{s}}\end{pmatrix}e^{\mbox{$-\frac{1}{4s}(r^2+z^2)$}}\right\|_{\LL{1}} ds
\\& \qquad 
+ \int_{t/2}^tC(t-s)^{-\frac{1}{2}}\uus{}\wwsp{}{q}\,ds 
\\& \qquad 
+ \int_{t/2}^t(1+s)^{-1}s^{-\frac{1}{2}}\,ds \left\|G_2\left(\frac{\cdot}{\sqrt{t}}\right)\right\|_{\LL{q}} t^{-\frac{5}{2}}.
\end{align*}
Using again estimate \eqref{majorationJ'} and then finally \eqref{uLinftyr2}, \eqref{lemmamajorationsr2w}, lemma \ref{lemmamajorationr1alpha+ralphaz} and lemma \ref{lemmamajorationr2alpha+ralphaz2}, one can state that
\begin{align*}
\|\widetilde{\w}(t)\|_{\LL{q}} 
& 
\leqslant C t^{\frac{1}{q}-\frac{5+\alpha}{2}} + C\left(\frac{t}{2}\right)^{\frac{1}{q}-\frac{5+\alpha}{2}} \int_0^{t/2}(1+s)^{\frac{\alpha}{2}-1}s^{-\frac{1}{2}}\,ds
\\& \qquad 
+ C\left(\frac{t}{2}\right)^{\frac{1}{q}-\frac{5+\alpha}{2}} \int_0^{t/2} \frac{(1+s)^{-1}s^{-\frac{1}{2}}}{s^2}s^{\frac{4+\alpha}{2}}\,ds
\\& \qquad 
+ C\int_{t/2}^t(t-s)^{-\frac{1}{2}}s^{-\frac{2+\alpha}{2}}s^{\frac{1}{q}-2}\,ds 
\\& \qquad 
+ C\int_{t/2}^ts^{-\frac{2+\alpha}{2}}\,ds\ t^{\frac{1}{q}}t^{-\frac{5}{2}}
\\&
\leqslant C\,t^{\frac{1}{q}-\frac{5}{2}-\frac{\alpha}{2}}
\end{align*}
for any positive time $t$, where the last constant depends on the initial data $\wo$ and on the quantity $\frac{1}{\alpha}$.
\end{proof}

\section{Additional Results and concluding remarks}

\subsection{The First Resonant term}

In the previous section we have shown that the first terms in the vorticity's asymptotic expansion decrease in $L^p$ as negative powers of $t$. This had already been highlighted in \cite{Carpio94,GalWayR3}, or in the case of the velocity by the works of Wiegner \cite{Wiegner87}, Carpio \cite{Carpio96}, Miyakawa and Schonbek \cite{Schonbek85,Schonbek86,
FujigakiMiyakawa2001,MiyakawaSchonbek2001} and more recently Brandolese \cite{Brandolese2004,Brandolese2003} for instance. However, when addressing the higher order asymptotics, Gallay and Wayne observe in \cite{GalWayR2,GalWayR3} what they call resonances: interactions between the eigenvalues of the linear operator $\Lambda = \Delta + \frac{1}{2}x\cdot\nabla+1$, which is the one that appears when they consider the vorticity equation \eqref{vor} in the self-similar variables $(\ln(1+t),\frac{x}{\sqrt{1+t}})$. In particular, in the two-dimensional case, they show explicitly the presence of a $\ln(t)$ factor in the third-order asymptotic term. This is also something we observe. When $(r,z)\mapsto (r^4+r^2z^2+rz^2)\wo(r,z)$ is in $\LL{1}$ appears the first tangible manifestation of the nonlinearity, namely a bound that does not decreases merely as a negative power of $t$. This is the reason why, in general, proposition \ref{propwr3+alpha} does not hold for $\alpha=1$ and has to be replaced by the following result.

\begin{proposition}\label{propwr4}
Take $\w$ the solution of \eqref{vor} with initial condition $\wo\in\LL{1}$.\\
If $(r,z)\mapsto (r^4+r^2z^2+rz^2)\wo(r,z)$ is in $\LL{1}$ then there exists a positive constant $C_{\wo}$ such that for every $q\in[1,+\infty]$ and all $t>0$ 
\begin{equation}
\|\widetilde{\w}(t)\|_{\LL{q}} \leqslant C_{\wo}(1+\ln(1+t))\,t^{\frac{1}{q}-3}
\end{equation}
where
$\widetilde{\w}$ is given by \eqref{defwtilde}.
\end{proposition}

\begin{proof}
The proof is identical to the one of proposition \ref{propwr3+alpha}, with the only difference that having $\alpha=1$ in the last calculation leads to consider the integrals $\int_0^{t/2}(1+s)^{-\frac{1}{2}}s^{-\frac{1}{2}}ds$ and $\int_0^{t/2}(1+s)^{-1}ds$ which grow towards infinity like $\ln(1+t)$.
\end{proof}

This article ends before the study of the third-order asymptotic term, which is the first resonant term in the long-time expansion of the three-dimensional fluid.

\subsection{Decay rates for the velocity}

To deduce theorem \ref{theoremintrou} from what has been seen in the previous section, as well as some decay rates for the velocity, one has to invoke estimate \eqref{28} and the fact that the velocity fields corresponding to the vorticities $(t,r,z)\mapsto I_0G_1(\frac{r}{\sqrt{t}},\frac{z}{\sqrt{t}})t^{-2}$ and $(t,r,z)\mapsto J(t)G_2(\frac{r}{\sqrt{t}},\frac{z}{\sqrt{t}})t^{-5/2}$ are respectively $(t,r,z)\mapsto I_0u^{G_1}(\frac{r}{\sqrt{t}},\frac{z}{\sqrt{t}})t^{-3/2}$ and $(t,r,z)\mapsto J(t)u^{G_2}(\frac{r}{\sqrt{t}},\frac{z}{\sqrt{t}})t^{-2}$, where $u^{G_1}$ and $u^{G_2}$ are defined in \eqref{defuG1}.
Propositions \ref{propmajorationsralpha}, \ref{propwr2+alpha}, \ref{propwr3+alpha} and \ref{propwr4} then lead to the following result, where $I_0$ and $J(t)$ denote respectively $\int_{\Omega}r^2\rot(u_0(r,z))drdz$ and $\int_{\Omega}r^2z\,\rot(u(t,r,z))drdz$.

\begin{corollary}\label{corollaryuq}
Take $u$ the solution of \eqref{NSaxy} with initial data $u_0$ such that $\rot(u_0)\cdot e_\theta = \wo\in\LL{1}$. Take $\beta\geqslant0$.
\begin{itemize}
\item[(i)] If $\beta\in\ ]\frac{1}{2},2]$ and $(r,z)\mapsto r^\beta\wo(r,z)$ is in $\LL{1}$ then for every $q\in[2,+\infty]$ there exists a positive constant $C$ such that for all time $t>0$
\begin{equation}
\left\|u(t)\right\|_{L^q(\RR^3)} \leqslant C\,t^{\frac{3}{2q} - \frac{1}{2} - \frac{\beta}{2}}.
\end{equation}
\item[(ii)] If $\beta\in\ ]2,3]$ and $(r,z)\mapsto (r^\beta+r^{\beta-1}|z|+r^{\beta-2}|z|)\,\wo(r,z)$ is in $\LL{1}$ then for every $q\in[2,+\infty]$ there exists a positive constant $C$ such that for all time $t>0$
\begin{equation}
\left\|u(t) - I_0u^{G_1}\left(\frac{\cdot}{\sqrt{t}}\right)t^{-\frac{3}{2}}\right\|_{L^q(\RR^3)} \leqslant C\,t^{\frac{3}{2q} - \frac{1}{2} - \frac{\beta}{2}}.
\end{equation}
\item[(iii)] If $\beta\in\ ]3,4[$ and $(r,z)\mapsto (r^\beta+r^{\beta-2}z^2+r^{\beta-3}z^2)\,\wo(r,z)$ is in $\LL{1}$ then for every $q\in[2,+\infty]$ there exists a positive constant $C$ such that for all time $t>0$
\begin{equation}
\left\|u(t) - I_0u^{G_1}\left(\frac{\cdot}{\sqrt{t}}\right)t^{-\frac{3}{2}} - J(t)u^{G_2}\left(\frac{\cdot}{\sqrt{t}}\right)t^{-2}\right\|_{L^q(\RR^3)} \leqslant C\,t^{\frac{3}{2q} - \frac{1}{2} - \frac{\beta}{2}}.
\end{equation}
\item[(iv)] If $(r,z)\mapsto (r^4+r^2z^2+rz^2)\,\wo(r,z)$ is in $\LL{1}$ then for every $q\in[2,+\infty]$ there exists a positive constant $C$ such that for all time $t>0$
\begin{equation}
\left\|u(t) - I_0u^{G_1}\left(\frac{\cdot}{\sqrt{t}}\right)t^{-\frac{3}{2}} - J(t)u^{G_2}\left(\frac{\cdot}{\sqrt{t}}\right)t^{-2}\right\|_{L^q(\RR^3)} \leqslant C\ln(1+t)\,t^{\frac{3}{2q} - \frac{5}{2}}.
\end{equation}
\end{itemize}
\end{corollary}

\begin{proof}
When $q=+\infty$, one only has to apply inequality \eqref{26} together with the propositions named right above, the $L^\infty$-norms being rigorously the same for $u(t)$ on $(\RR^3,rdrd\theta dz)$ and $(\Omega,drdz)$. 

When $2\leqslant q<\infty$, a straightforward generalisation of those propositions is required. Let us consider $\alpha\in[0,1]$ such that $\alpha\leqslant\beta$ and $\frac{1}{q}\leqslant\alpha<\frac{1}{q}+1$. Estimate \eqref{28} with $\frac{1}{p} = \frac{1}{q} + \frac{1+\frac{1}{q} -\alpha}{2}$ gives that
\begin{equation}\label{28R3LOmega}
\|f(t)\|_{L^q(\RR^3)} = (2\pi)^{\frac{1}{q}}\ \|r^{\frac{1}{q}}f(t)\|_{\LL{q}} \leqslant C \|r^\alpha \rot f(t)\|_{\LL{p}}
\end{equation}
where $f(t)$ is $u(t)$, $u(t) - I_0u^{G_1}(\frac{\cdot}{\sqrt{t}})t^{-3/2}$ or $u(t) - I_0u^{G_1}(\frac{\cdot}{\sqrt{t}})t^{-3/2} - J(t)u^{G_2}(\frac{\cdot}{\sqrt{t}})t^{-2}$ depending on the case. Estimate \eqref{28} requires that $p$ is not chosen equal to 1, which is possible by taking $\alpha > \frac{3}{q}-1$ given that $\beta > \frac{3}{q}-1$ and that $\frac{1}{q}+1 > \frac{3}{q}-1$. 
Finally, the proof is achieved by using that
\begin{itemize}[label=$\circ$] 
\item $\wwtp{r^\alpha}{p} \leqslant C\,t^{\frac{1}{p}-1+\frac{\alpha-\beta}{2}}$ when $\beta\leqslant2$ (by interpolation from proposition \ref{propmajorationsralpha}),
\item $\|r^\alpha( \w(t) - I_0G_1(\frac{\cdot}{\sqrt{t}})t^{-2})\|_{\LL{p}} \leqslant C\,t^{\frac{1}{p}-1+\frac{\alpha-\beta}{2}}$ when $2<\beta\leqslant3$,
\item $\|r^\alpha\widetilde{\w}(t)\|_{\LL{p}} \leqslant C\,t^{\frac{1}{p}-1+\frac{\alpha-\beta}{2}}$ when $3<\beta<4$,
\item and $\|r^\alpha\widetilde{\w}(t)\|_{\LL{p}} \leqslant C\ln(1+t)\,t^{\frac{1}{p}-3+\frac{\alpha}{2}}$ when $\|(1+r^4+r^2z^2+rz^2)\wo\|_{\LL{1}} < \infty$;
\end{itemize}
the demonstrations of the last three inequalities being the same calculations as for propositions \ref{propwr2+alpha} and \ref{propwr3+alpha} as long as $\alpha\in[0,1]$.
\end{proof}

Eventually, let us prove theorem \ref{theoremintrou}.

\begin{proof}[Proof of theorem \ref{theoremintrou}]
Let us keep the same notations as in the previous proof, considering here only the cases when $\beta=2$ and $f(t) = u(t) - I_0u^{G_1}(\frac{\cdot}{\sqrt{t}})t^{-3/2}$ or $\beta=3$ and $f(t) = u(t) - I_0u^{G_1}(\frac{\cdot}{\sqrt{t}})t^{-3/2} - J_\infty u^{G_2}(\frac{\cdot}{\sqrt{t}})t^{-2}$. Propositions \ref{propwr2} and \ref{propwr3} are still true with a radial weight $r^\alpha$ when $0\leqslant\alpha\leqslant2$, meaning that 
\begin{equation}\label{local9}
t^{2-\frac{1}{p}-\frac{\alpha}{2}}\left\|r^\alpha\left( \w(t,r,z) - I_0G_1\left(\frac{\cdot}{\sqrt{t}}\right)t^{-2}\right)\right\|_{\LL{p}} \xrightarrow[t\to+\infty]{} 0
\end{equation}
for all $p\in[1,+\infty]$ if $(1+r^2)\wo$ is in $\LL{1}$ and
\begin{equation}\label{local10}
t^{\frac{5}{2}-\frac{1}{p}-\frac{\alpha}{2}}\left\|r^\alpha\left( \w(t,r,z) - I_0G_1\left(\frac{\cdot}{\sqrt{t}}\right)t^{-2} - J(t)G_2\left(\frac{\cdot}{\sqrt{t}}\right)t^{-\frac{5}{2}}\right)\right\|_{\LL{p}} \xrightarrow[t\to+\infty]{} 0
\end{equation}
for all $p\in[1,+\infty]$ if $(1+r^3+r^2|z|+r|z|)\wo$ is in $\LL{1}$. The demonstrations of \eqref{local9} and \eqref{local10} are similar to the original ones (respectively \eqref{propwr2convergence} and \eqref{propwr3convergence}). Inequality \eqref{28R3LOmega} can thus be applied with $\frac{3}{q}-1<\alpha\leqslant2$, so $q$ can be taken in all $]1,+\infty[$\,. This concludes the proof, remembering that the case $q=+\infty$ is directly included in propositions \ref{propwr2} and \ref{propwr3} and remembering remark \ref{rkJinfty} about the presence of $J_\infty$ in \eqref{theoremintrour3}.
\end{proof}

\subsection{Axisymmetry with swirl}

When the swirl is nonzero, the information of the Navier-Stokes equations is contained in the pair $(u_\theta,\w)$ through the vorticity-stream formulation 
\[
\left\{\begin{array}{r<{\hskip-0.6em}l}
\partial_tu_\theta + u_r\partial_ru_\theta + u_z\partial_zu_\theta
&\displaystyle
= \Delta u_\theta - \frac{1}{r^2}u_\theta - \frac{u_r}{r}u_\theta\\
\partial_t\w + u_r\partial_r\w + u_z\partial_z\w
&\el{1.7em} \displaystyle
= \Delta\w - \frac{1}{r^2}\w + \frac{u_r}{r}\w + \frac{1}{r}\partial_z(u_\theta^2)
\\\displaystyle
-\left(\frac{1}{r}\Delta-\frac{2}{r^2}\partial_r \right) \psi
&\el{1.7em}
= \w 
\end{array}\right.
\]
where $\psi$ is the stream function, $u = -\frac{1}{r}\partial_z\psi\,e_r + u_\theta\,e_\theta + \frac{1}{r}\partial_r\psi\,e_z$ and $\omega = -\partial_zu_\theta\,e_r + \w\,e_\theta + \frac{1}{r}\partial_r(ru_\theta)e_z$, see for example. The tangential vorticity $\w$ is thus linked to the velocity in the axial plane $(u_r,u_z)$, and the swirl $u_\theta$ is linked to the vorticity $(\omega_r,\omega_z)$ in the axial plane.

The first and second order asymptotic terms can still be calculated, for example from the general three-dimensional formulae given by Gallay and Wayne in \cite{GalWayR3}, but the vorticity has to be assumed continuous in time into the weighted space $L^2(\RR^3,(1+|x|)^mdx)^3$ for some $m>\frac{7}{2}$. Let us suppose this is the case. The vorticity's expansion in \cite{GalWayR3} then shows that the first two asymptotics of the tangential vorticity $\w$ are the same as when the swirl is zero:
\[
\w(t,r,z) = \frac{rI_0}{16\sqrt{\pi}}e^{\mbox{$-\frac{r^2+z^2}{4t}$}}t^{-5/2} + \frac{rzJ(t)}{32\sqrt{\pi}} e^{\mbox{$-\frac{r^2+z^2}{4t}$}} t^{-7/2} + \underset{t\to+\infty}{o}\left(t^{-7/4}\right)
\]
in $L^2(\RR^3)$, where $I_0 = \int_\Omega r^2\w(0,r,z)drdz$ and $J(t) = \int_\Omega r^2z\,\w(t,r,z)drdz$. The radial and vertical vorticities, however, present now nontrivial terms which contribute to the second-order asymptotics as
\begin{align*}
&\omega_r(t,r,z) = \frac{rz(K_0+L_0)}{32\sqrt{\pi}} e^{\mbox{$-\frac{r^2+z^2}{4t}$}} t^{-7/2} + \underset{t\to+\infty}{o}\left(t^{-7/4}\right)
\\
&\omega_z(t,r,z) = \frac{(4t-r^2)(K_0+L_0)}{32\sqrt{\pi}} e^{\mbox{$-\frac{r^2+z^2}{4t}$}} t^{-7/2} + \underset{t\to+\infty}{o}\left(t^{-7/4}\right)
\end{align*}
in $L^2(\RR^3)$, where $K_0 = \int_\Omega r^2z\,\omega_r(0,r,z)drdz$ and $L_0 = \int_\Omega r(2-z^2)\,\omega_z(0,r,z)drdz$. 

A natural arising question is whether it would be easy to follow a similar approach to what has been done in the present paper when the swirl is not supposed zero, in order to weaken the hypotheses taken on $\omega$ by \cite{GalWayR3}.

\vspace{0.45cm} 

\section*{Acknowledgment}

Many thanks to Thierry Gallay for his support, his advice, his comments and suggestions.
The author wants to thank as well the reviewer for their detailed feedback and their helpful remarks.

\appendix

\section{Appendix: On The Hypotheses of theorem \ref{theoremintro}}\label{annexecomparaisonThGWay3D}

We come back in this appendix to the comparison between the hypotheses we made for the second point of our theorem \eqref{theoremintror3} and those needed for the second-order asymptotic development from \cite{GalWayR3}. Applying theorem 4.2 of \cite{GalWayR3} in the zero-swirl axisymmetric case requires that the initial vorticity $\omega(0) = \wo e_\theta$ satisfies $\woop{(r^{\frac{1}{2}}+r^{m+\frac{1}{2}}+r^{\frac{1}{2}}|z|^m)}{2} < \infty$  for some $m>\frac{7}{2}$. This implies in particular that $\wwtp{(r^{\frac{1}{2}}+r^{m+\frac{1}{2}}+r^{\frac{1}{2}}|z|^m)}{2} < \infty$ and $\wwtp{(1+r^{m}+|z|^m)}{\infty} < \infty$ for all positive time $t>0$ (see \cite[prop 2.5]{GalWayR3}); we can therefore consider any instant $t_0>0$ as the initial time and see that our assumption (considered at $t_0>0$) is a consequence of that.

What we assume in theorem \ref{theoremintro} is that $\wwtt{(1+r^3+r^2|z|+r|z|)}{t_0}$ is finite. Using Hölder's inequality to write
\[
\int_\Omega |r^2z\,\w(t_0,r,z)|\,drdz \leqslant \left(\int_\Omega \frac{r^3\,dr}{(1+r)^{\frac{8}{7}m}}\frac{z^2\,dz}{(1+|z|)^{\frac{6}{7}m}}\right)^\frac{1}{2} \wwttp{r^{\frac{1}{2}}(1+r)^{\frac{4}{7}m}(1+|z|)^{\frac{3}{7}m}}{t_0}{2},
\]
then Young's inequality to write $(1+r)^{\frac{4}{7}m}(1+|z|)^{\frac{3}{7}m}|\w(t_0)| \leqslant \frac{4}{7}(1+r)^m|\w(t_0)|+\frac{3}{7}(1+|z|)^m|\w(t_0)|$ and using finally \eqref{deltasum} with $\delta = m$, one concludes that $\wwtt{r^2z}{t_0} \leqslant C \wwttp{(r^{\frac{1}{2}}+r^{m+\frac{1}{2}}+r^{\frac{1}{2}}|z|^m)}{t_0}{2}$ where the constant $C$ only depends on $m > \frac{7}{2}$. Similar computations hold for $\wwtt{r^3}{t_0}$ and $\wwtt{rz}{t_0}$, however $\wwtt{}{t_0}$ has to be cut in two parts. Indeed, one has to write using Hölder's inequality that
\[
\int_{[1,+\infty[\times\RR} |\w(t_0,r,z)|\,drdz \leqslant \left(\int_{[1,+\infty[\times\RR} \frac{r^{-1}\,drdz}{r^m(1+|z|)^m}\right)^\frac{1}{2} \wwttp{r^{\frac{1}{2}m+\frac{1}{2}}(1+|z|)^{\frac{1}{2}m}}{t_0}{2}
\]
and
\[
\int_{]0,1]\times\RR} |\w(t_0,r,z)|\,drdz \leqslant \left(\int_{]0,1]\times\RR} \frac{drdz}{(1+|z|)^m}\right) \|(1+|z|)^m\w(t_0)\|_{L^\infty(\Omega)}
\]
to get that $\w(t_0)$ belongs to $\LL{1}$.

Our assumption is thus weaker than the one from \cite{GalWayR3} for any positive time. Let us notice finally that due to the symmetry of our study we do not need every moment of order three of $\wo$ to be finite, while $\woop{(r^{\frac{1}{2}}+r^{m+\frac{1}{2}}+r^{\frac{1}{2}}|z|^m)}{2}$ involves all the moments of order $m$ of $\wo$.

\section{Appendix: A Fixed Point argument}\label{annexeprop4.1}

The goal of this appendix is to show that under the hypothesis of $(r,z)\mapsto(1+r)\wo(r,z)$ belonging to $\LL{1}$, there exists a time $T>0$ such that the solution $t\mapsto\w(t)$ to the integral vorticity equation \eqref{duhamel} is continuous from $[0,T]$ into the space $L^1_1(\Omega) = \{w:\Omega\to\RR \mid \|w\|_{\LL{1}} + \|rw\|_{\LL{1}}<\infty\}$.

The strategy developed is similar to the two-dimensional case, see for example \cite{Benartzi94,Kato94}. We rewrite here, up to the appropriate adaptations, the fixed-point argument that can be found in \cite[prop. 4.1]{Gallay1}.
Let $\wo$ belong to $L^1_1(\Omega)$. Define for any $T>0$ the function space
\[
X_T = \left\{\w\in\CC^0\!\left(]0,T],\LL{4/3}\right) \,\middle|\ \|\w\|_{X_T}<\infty\right\}
\]
equipped with the norm
\[
\|\w\|_{X_T} = \underset{0<t\leqslant T}{\sup}\left(t^{1/4} \wwtp{}{4/3} + t^{1/4}\wwtp{r}{4/3}\right).
\]
Let us see that, given \eqref{46}, there exists some positive constant $C>0$ independent of $T$ such that $\|S(\cdot)\wo\|_{X_T} \leqslant C\woo{(1+r)}$. In particular, 
$t\mapsto S(t)\wo$ belongs to $X_T$ whatever $T$ is. 
Moreover, it also follows from \eqref{46} that $\|S(\cdot)\wo\|_{X_T}\longrightarrow0$ as $T$ tends to 0 for the following reasons. By density, for any $\varepsilon>0$ there exists $\wo^\varepsilon$ such that $\|(1+r)\wo^\varepsilon\|_{\LL{4/3}} <\infty$ and $\|(1+r)(\wo-\wo^\varepsilon)\|_{\LL{1}} <\varepsilon$. Hence, using \eqref{46} twice with $\beta=\alpha=0$ and twice with $\beta=\alpha=1$,
\[\begin{array}{ll}
\underset{T\to0}{\lim}\,\|S(\cdot)\wo\|_{X_T} 
&\kern-0.6em \leqslant \displaystyle 
\underset{T\to0}{\lim}\ \underset{0<t\leqslant T}{\sup} \left(t^{1/4} \|S(t)[\wo-\wo^\varepsilon]\|_{\LL{4/3}} + t^{1/4}\|rS(t)[\wo-\wo^\varepsilon]\|_{\LL{4/3}} 
\right.\\\rule[1.3em]{0pt}{0pt} &\hfill \displaystyle \left.
+ t^{1/4}\|S(t)\wo^\varepsilon\|_{\LL{4/3}} + t^{1/4}\|rS(t)\wo^\varepsilon\|_{\LL{4/3}}\right)
\\\rule[1.7em]{0pt}{0pt} &\kern-0.6em \leqslant \displaystyle 
\underset{T\to0}{\lim}\ \underset{0<t\leqslant T}{\sup} C \left(\|\wo-\wo^\varepsilon\|_{\LL{1}} + \|r(\wo-\wo^\varepsilon)\|_{\LL{1}} \rule[1.1em]{0pt}{0pt}
\right.\\\rule[1.3em]{0pt}{0pt} &\hfill \displaystyle \left. 
+ t^{1/4}\|\wo^\varepsilon\|_{\LL{4/3}} + t^{1/4}\|r\wo^\varepsilon\|_{\LL{4/3}}\right)
\\\rule[1.7em]{0pt}{0pt} &\kern-0.6em \leqslant \displaystyle 
2\varepsilon
\end{array}\]
and so $\lim_{T\to0}\|S(\cdot)\wo\|_{X_T} = 0$. This will be useful latter.

Given $\w$ in $X_T$ and $p\in[1,\frac{4}{3}]$, 
and taking $u$ the velocity field obtained from $\w$ via the axisymmetric Biot-Savart law \eqref{BSaxi}, define the map 
\[
A\w : t \longmapsto \int_0^tS(t-s)\dii(u(s)\w(s))\,ds
\]
which is continuous from $]0,T]$ into the space $\{w:\Omega\to\RR\mid\|w\|_{\LL{p}} + \|rw\|_{\LL{p}}<\infty\}$ 
. Using \eqref{47} with $\beta=\alpha=0$, Hölder's inequality and then \eqref{25}, see that
\begin{equation}\label{54}
\begin{array}{ll}
t^{1-\frac{1}{p}}\|A\w(t)\|_{\LL{p}}
\kern-0.6em & \displaystyle
\leqslant t^{1-\frac{1}{p}}\int_0^tC(t-s)^{\frac{1}{p}-\frac{3}{2}}\|u(s)\w(s)\|_{\LL{1}}\,ds
\\ \el{2.2em}&\displaystyle
\leqslant Ct^{1-\frac{1}{p}}\int_0^t(t-s)^{\frac{1}{p}-\frac{3}{2}}\|u(s)\|_{\LL{4}}\wwsp{}{4/3}\,ds
\\ \el{2.2em}&\displaystyle
\leqslant Ct^{1-\frac{1}{p}}\int_0^t(t-s)^{\frac{1}{p}-\frac{3}{2}}\wwsp{}{4/3}^2\,ds
\\ \el{2.2em}&\displaystyle
\leqslant Ct^{1-\frac{1}{p}}\int_0^t(t-s)^{\frac{1}{p}-\frac{3}{2}}s^{-\frac{1}{2}}\|\w\|_{X_T}^2\,ds
\\ \el{1.7em}&\displaystyle
\leqslant C\|\w\|_{X_T}^2
\end{array}
\end{equation}
for $t$ in $]0,T]$. The same computation gives that $t^{1-\frac{1}{p}}\|rA\w(t)\|_{\LL{p}} \leqslant C\|\w\|_{X_T}^2$, using \eqref{47} with $\beta=\alpha=1$. This entails, by choosing $p=\frac{4}{3}$, that $A\w$ belongs to $X_T$ and that $\|A\w\|_{X_T} \leqslant C_1\|\w\|_{X_T}^2$ for some positive constant $C_1$. More generally, the Lipschitz estimate
\begin{equation}\label{55}
\|A\w - A\w'\|_{X_T} \leqslant C_1 (\|\w\|_{X_T} + \|\w'\|_{X_T}) \|\w-\w'\|_{X_T}
\end{equation}
holds for every $\w,\w'$ in $X_T$, stemming again from computations similar to \eqref{54}.

Consider now the operator $\widetilde{A} : X_T\to X_T : \w \mapsto S(\cdot)\wo - A\w$. Fix $R>0$ such that $2C_1R <1$ and denote by $B_R$ the closed ball of radius $R$ centered at the origin in $X_T$. One can now chose $T>0$ small enough for $\|S(\cdot)\wo\|_{X_T} \leqslant \frac{1}{2}R$ to hold, so that $\widetilde{A}$ maps $B_R$ into $B_R$ and is a strict contraction there and thus has a unique fixed point in $B_R$ thanks to the Banach fixed point theorem. By construction, this unique fixed point $\w$ is the unique solution of \eqref{duhamel}, which we already know. 

To achieve the proof, it only remains to show that $t\mapsto A\w(t)$ is continuous at 0 for the norm $\|w\|_{L^1_1(\Omega)} = \|w\|_{\LL{1}} + \|rw\|_{\LL{1}}$. 
Indeed, $(S(t))_{t\geqslant0}$ is a strongly continuous semigroup on $L^1_1(\Omega)$ and we already said that $A\w$ is continuous from $]0,T]$ into $L^1_1(\Omega)$. Let us note that $\|\w\|_{X_T} \longrightarrow 0$ as $T$ goes to zero, since one can take $T>0$ as small as desired and, from the moment that $T$ is little enough, 
the Banach fixed point theorem applies in the ball $B_R$ of radius $R=2\|S(\cdot)\wo\|_{X_T}$ which tends to zero as seen above. Consequently, remembering that $t^{1-\frac{1}{p}}\|(1+r)A\w(t)\|_{\LL{p}} \leqslant C\|\w\|_{X_T}^2$ for all $t\in\ ]0,T]$ and taking $p=1$ implies that $\|A\w(t)\|_{L^1_1(\Omega)}\longrightarrow0$ when $t$ goes to zero. 
Therefore $\w = S(\cdot)\wo - A\w$ is continuous from $[0,T]$ into $L^1_1(\Omega)$.

\refstepcounter{theorem}
\paragraph{Remark \thetheorem.}\label{rkC1}
These arguments can be extended to show that $\w$ is also differentiable from $]0,T]$ into $L^1_1(\Omega)$. This results from the analyticity of the semigroup $(S(t))_{t\geqslant0}$ and the fact that $A$ defined above is a locally Lipschitz operator (see \eqref{55}), as developed for example in \cite[§3.3]{DanHenry}. The same way, when saying in the proof of lemma \ref{lemmamajorationsr2} that $\w\in\CC^0([0,+\infty[,L^1_2(\Omega))$ from an initial data $\wo\in L^1_2(\Omega) = L^1(\Omega,(1+r^2)drdz)$, we can show that $\w\in\CC^1(]0,+\infty[,L^1_2(\Omega))$. The same is true as well on $L^1(\Omega,(1+r^2+r^2|z|)drdz)$.

\section{Appendix: A Grönwall lemma}\label{annexeGronwallameliore}

This appendix tackles the boundedness of a nonnegative function $f$ satisfying 
\begin{equation}\label{gronwallameliorehypothesis}
\forall t\in[0,T[ \qquad f(t) \leqslant C + C\int_0^t(t-s)^{\beta-1}(1+s)^{-\gamma}f(s)\,ds
\end{equation}
where $0<\beta<\gamma$ and $0<T\leqslant+\infty$. In particular, satisfying this inequality for all $t\in[0,T=+\infty[$ ensures that $f$ belongs to $L^\infty([0,+\infty[,[0,+\infty[)$.

The hypothesis made here \eqref{gronwallameliorehypothesis} differs from the usual sufficient condition asked for Grönwall's inequality by the presence of the factor $(t-s)^{\beta-1}$ in the integrand. This hypothesis is partially covered in the Dan Henry's book \cite{DanHenry}, but lemmas 7.1.1 and 7.1.2 from \cite{DanHenry} do not not rigorously include the cases we need here.

\begin{lemma}\label{lemmaGronwallbis-gamma}
Consider $f$ in $L^\infty_{loc}([0,T[,[0,+\infty[)$ where $0<T\leqslant+\infty$. If there are $a\geqslant0$, $b\geqslant0$ and $\gamma>\beta>0$ such that
\begin{equation}
\forall t \in[0,T[\qquad f(t) \leqslant a + b\int_0^t (t-s)^{\beta-1}(1+s)^{-\gamma}f(s)\,ds
\end{equation}
with $\beta\leqslant1$, then there exists a positive constant $C_{b,\beta,\gamma}$ such that $\|f\|_{L^\infty([0,T[)} \leqslant a\,C_{b,\beta,\gamma}$.
\end{lemma}

\begin{proof}
Let us first exhibit one dominating function for $f$ that does not depend on $f$ itself. This is lemma 7.1.1 of \cite{DanHenry}, if we use that trivially $(1+s)^{-\gamma}\leqslant1$. Denoting by $A$ the linear operator
\[
Ag : t \longmapsto b\int_0^t (t-s)^{\beta-1}(1+s)^{-\gamma} g(s)\,ds
\]
on $L^{\infty}_{\text{loc}}([0,T[)$, and remembering that $B$ denotes the Euler Beta function, a recurrence gives that on $[0,T[$
\[
\forall n\in\NN \qquad f(t) \leqslant a\sum_{k=0}^{n} c_k t^{k\beta} + A^{n+1}f(t)
\]
where $c_0 = 1, c_{n+1} = c_nbB(\beta,n\beta+1)$ and
\[
\forall n\in\NN^* \qquad A^nf(t) \leqslant b\,d_n\int_0^t (t-s)^{n\beta-1} f(s)\,ds
\]
where $d_1 = 1, d_{n+1} = d_nbB(\beta,n\beta)$. Given that $B(\beta,n\beta+1)$ and $B(\beta,n\beta)$ are both $O(n^{-\beta})$ when $n$ tends to infinity, $A^nf(t)$ tends to zero for every $t$ in $[0,T[$ and $\sum c_k t^k$ has an infinite radius of convergence. Thus
$f(t) \leqslant a \sum_{k\geqslant0}c_kt^{k\beta}$ for all $t$ in $[0,T[$.

Let us now give a uniform bound to $f$ on $[0,T[$ which does not depend on $T$. Set $t_0 = (2bB(\beta,1-\beta))^{\frac{1}{\gamma-\beta}}$ if $\beta < 1$, and $t_0 = (\frac{2b}{\gamma-\beta})^{\frac{1}{\gamma-\beta}}ds$ if $\beta = 1$. Let us see that in the first case
\[
\int_{t_0}^t(t-s)^{\beta-1}(1+s)^{-\gamma}ds \leqslant \int_{t_0}^t(t-s)^{\beta-1}s^{-\beta}t_0^{\beta-\gamma}\,ds \leqslant B(\beta,1-\beta)\,t_0^{\beta-\gamma},
\]
and that
\[
\int_{t_0}^t(t-s)^{\beta-1}(1+s)^{-\gamma}ds \leqslant \int_{t_0}^t(t-s)^{\beta-1}s^{-\beta}t_0^{\beta-\gamma}\,ds \leqslant \frac{1}{\gamma-\beta}t_0^{\beta-\gamma}
\]
if $\beta=1$. Therefore, in both cases, either $T\leqslant t_0$ and obviously $\|f\|_{L^{\infty}([0,T[)} \leqslant a\sum_{k\geqslant0}c_kt_0^{k\beta}$, or $T>t_0$ and 
\begin{align*}
\|f\|_{L^{\infty}([t_0,T[)}
&\leqslant  a + b \int_0^{t_0}(t-s)^{\beta-1}(1+s)^{-\gamma}f(s)\,ds + b \int_{t_0}^t(t-s)^{\beta-1}(1+s)^{-\gamma}f(s)\,ds \\
&\leqslant  a + b \int_0^{t_0} (t_0-s)^{\beta-1}\, \|f\|_{L^{\infty}([0,t_0])} \,ds + \frac{1}{2}\ \|f\|_{L^{\infty}([t_0,T[)}
\end{align*}
so
\[
\|f\|_{L^{\infty}([t_0,T[)} \leqslant 2a + 2b\,\frac{1}{\beta}t_0^\beta\,\|f\|_{L^{\infty}([0,t_0])} \leqslant 2a + 2ab\,\frac{1}{\beta}t_0^\beta\,\sum_{k\geqslant0}c_kt_0^{k\beta}
\]
and $\|f\|_{L^{\infty}([0,T[)} \leqslant \|f\|_{L^{\infty}([0,t_0])} + \|f\|_{L^{\infty}([t_0,T[)} \leqslant aC_{b,\beta,\gamma}$.
\end{proof}

The crucial point of this lemma is that the constant $C_{b,\beta,\gamma}$ can be chosen independently of $T$. The restriction $\beta \leqslant 1$ can be easily removed, but the proof would be longer.

\bibliographystyle{plain}
\bibliography{C:/Users/quent/Documents/Softwares/Latex/_mon_test_de_LaTex/BibQV}

\end{document}